\documentclass[preprint,10pt]{elsarticle}
\usepackage{amssymb}
\usepackage{amsthm}
\usepackage{graphics}
\usepackage{amsmath}
\usepackage[dvipsnames,usenames]{color}

\newtheorem{thm}{Theorem}[section]

\newtheorem{ass}[thm]{Assumption}
\newtheorem{rem}{Remark}

\newtheorem{exa}[thm]{Example}
\theoremstyle{definition}

\numberwithin{equation}{section}

\journal{Communications in Nonlinear Science and Numerical Simulation}

\begin{document}
	
	\begin{frontmatter}
		
		
		
		\title{Propagation of chaos in infinite horizon and numerical stability for  stochastic McKean-Vlasov equations}

		\author[label1]{Zhuoqi Liu}
		\author[label2]{Shuaibin Gao}
		\author[label3]{Chenggui Yuan}
		\author[label1]{Qian Guo\corref{cor1}}\ead{qguo@shnu.edu.cn}
		
		\cortext[cor1]{Corresponding author at: Department of Mathematics, Shanghai Normal University, Shanghai 200234, China.}
		\address[label1]{Department of Mathematics, Shanghai Normal University, Shanghai 200234, China}
		\address[label2]{School of Mathematics and Statistics, South-Central Minzu University,
			Wuhan 430074, China}
		\address[label3]{Department of Mathematics, Swansea University, Bay Campus, Swansea SA1 8EN, UK}

		\begin{abstract}
			This paper focuses on the numerical stability of stochastic McKean-Vlasov equations (SMVEs) via the stochastic particle method. Firstly, the long-time propagation of chaos in the mean-square sense is obtained, and the almost sure propagation in infinite horizon is also proved.
			Next, when the coefficients satisfy linear growth conditions, the mean-square and almost sure	exponential stabilities of the Euler-Maruyama (EM) scheme associated with the corresponding interacting particle system are shown through an ingenious manipulation of empirical measure. Then, for the case that the state variables in drift and diffusion are both superlinear, the mean-square exponential stability of the backward EM scheme for the interacting system is achieved without the particle corruption, which is a novel conclusion. Moreover, under the linear growth condition on the diffusion coefficient, the almost sure stability of the backward EM scheme is studied. Combining these assertions enables the numerical solutions to reproduce the stabilities of the original SMVEs. The examples, including a feedback control problem and a stochastic opinion dynamics model, are provided to demonstrate  the importance of theoretical analysis of numerical stability.
		\end{abstract}
		
		\begin{keyword}
			Stability; Stochastic McKean-Vlasov equations;  Propagation of chaos; Euler-Maruyama scheme; backward Euler-Maruyama scheme
		\end{keyword}
		
	\end{frontmatter}

	% \linenumbers
	

	The applications of stochastic McKean-Vlasov equations (SMVEs) have been presented in many scientific fields such as biological models, mathematical finance, and physics, which leads to widespread research by scholars \cite{dosjia,4,5}. After the pioneering works by McKean in \cite{6,8}, the existence, uniqueness, and moment boundedness of SMVEs have been well-established \cite{3,21,32}. Moreover, the results, including Harnack inequalities \cite{32}, Feynman-Kac formula \cite{2}, gradient estimate \cite{suoy}, and ergodicity \cite{11} have been investigated.
	
	In general, the numerical solutions are investigated to help scholars obtain the properties of analytical solutions to SMVEs, which are difficult to be solved explicitly. The stochastic particle method in \cite{1,0} is a well-known scheme to approximate SMVEs. The main idea is
	to use the empirical measure
	of the corresponding interacting particle system to approach the law of original SMVE, then the numerical schemes can be established to approximate such interacting particle system. By means of the stochastic particle method, many scholars have discussed the numerical schemes, such as tamed Euler-Maruyama (EM) scheme \cite{31,9}, split-step scheme \cite{01}, multi-level Monte-Carlo scheme \cite{bao00}, and so on.
	To improve the convergence rate of numerical scheme to the interacting particle system, the tamed Milstein scheme was discussed in \cite{bao00,000}. By exploiting the fixed point theorem, the existence and uniqueness of the SMVEs whose drift and diffusion coefficients both can grow superlinearly were shown in \cite{51}, then the tamed EM and Milstein schemes were proposed. Moreover, the numerical schemes for SMVEs with irregular coefficients were also analyzed via the Yamada-Watanabe approximation technique and the Zvonkin transformation in \cite{001,05}.
	
	On the other hand, the infinite-time stability of stochastic system is crucial in some application scenarios. For the classical stochastic differential equations (SDEs) (whose coefficients are independent of the law), the moment stability and almost sure stability have been systematically investigated in \cite{100,mao0}.
	Then, these two types of stability of the numerical schemes for classical SDEs have been widely discussed, such as EM scheme \cite{mao4,MM,dengfeiqi}, $\theta$-EM scheme \cite{huangc,maom,wanggan,34}, truncated EM scheme \cite{songm,languang}, backward EM scheme \cite{pguo}, split-step EM scheme \cite{xiezhang}, adaptive EM scheme \cite{adaptiveEuler}, and so on.
	
	Returning to SMVEs, there are a few papers which analyze the stability of their analytical solutions. The exponential stability of the semilinear McKean-Vlasov stochastic evolution equation was shown in \cite{gov}. Then, the moment stability and almost sure stability of SMVEs \cite{qiaosiam} and multivalued SMVEs \cite{gongqiao} were presented.
	And the ergodicity for SMVEs was analyzed in \cite{axina3,axina1}.
	Based on \cite{qiaosiam}, the stabilization of SMVEs with feedback control via discrete time state observation was revealed in \cite{wuhao}, 
	where the stability of the corresponding interacting particle system was also analyzed due to the unobservability of the law of SMVEs. Unfortunately, although the example was provided in \cite{wuhao}, the corresponding numerical solution was not performed, and the theoretical analysis of the numerical scheme was not shown either.
	% Then our paper can fill this gap.

	Furthermore, many models have superlinear state variables in their coefficients, such as the FitzHugh-Nagumo model \cite{dosjia}, Cucker-Smale model \cite{axina2}, and so on.
	However, according to \cite{boom}, the moment of the unmodified EM schemes will diverge in finite time for SDEs with superlinear coefficients. As a result, implicit schemes have been widely used to handle such cases. 
	The convergence rate of the backward EM scheme for SMVE was analyzed in \cite{31}. 
	Then, it is necessary to analyze the mean-square and almost sure exponential stabilities of the numerical solution to the backward EM scheme.

	To sum up, based on the stochastic particle method,	this paper focuses on the infinite-time exponential stabilities of the numerical schemes for the interacting particle systems related to the original SMVEs. 
	Firstly, the mean-square and almost sure exponential stabilities of the interacting particle system are given.
	Next, the long-time propagation of chaos in the mean-square sense is proved. Since the almost sure propagation of chaos in infinite horizon is difficult to obtain directly, 
	the Chebyshev inequality and the Borel-Cantelli lemma are
	utilized to achieve the purpose.
	Then, the mean-square and almost sure exponential stabilities with rates of EM scheme for the interacting particle are revealed, under the assumption that the coefficients are linearly growing.
	Moreover, we also construct the backward EM scheme for the interacting particle system with superlinear coefficients, and analyze its mean-square and almost sure stabilities.
	Finally, some examples, including a feedback control problem and a stochastic opinion dynamics model,  are simulated to verify the effectiveness of the EM scheme and the backward EM scheme.

	%Next, the mean-square and almost sure exponential stabilities with rates of EM and backward EM scheme for the interacting particle are revealed.

	The paper is organized as follows.
	Some notations, assumptions, and particle systems are presented in Section 2.
	The mean-square and almost sure stabilities of the interacting particle system are discussed in Section 3.
	The theories about the propagation of chaos in infinite horizon are analyzed in Section 4.
	The mean-square and almost sure stabilities of the EM and backward EM schemes are presented in Sections 5 and 6, respectively.
	In the last section, we give several numerical examples to illustrate the theories.

	\section{Preliminary} 
	%\textbf{Notations:}
	Let $\big(\Omega,\mathcal{F},(\mathcal{F}_t)_{t\geq0},\mathbb{P}\big)$ be a complete probability space with the filtration $(\mathcal{F}_{t})_{t\geq0}$ that satisfies the usual conditions. Denote by $\mathbb{E}$ the probability expectation with respect to (w.r.t.) $\mathbb{P}$. Let $|\cdot|$ and $\|\cdot\|$ represent the Euclidean norm for the vector in $\mathbb{R}^d$ and the norm for the matrix in $\mathbb{R}^{d\times m}$, respectively.
	For $p\geq1$, the set of random variables $Z$ in $\mathbb{R}^d$ with $\mathbb{E}|Z|^p<\infty$ is  denoted by $L^p(\mathbb{R}^d)$. 
	%Let $c_1\wedge c_2=\min\{c_1,c_2\}$
	%and $c_1\vee c_2=\max\{c_1,c_2\}$ for real numbers $c_1, c_2$. 
	Let $\mathcal{P}(\mathbb{R}^d)$ be the collection of all probability measures on $\mathbb{R}^d$. For $p\geq1$, define 
	\begin{equation*}
		\mathcal{P}_{p}(\mathbb{R}^d)=\left\lbrace \mu\in\mathcal{P}(\mathbb{R}^d):\left(\int_{\mathbb{R}^d}|x|^p\mu(dx)\right)^{1/p}<\infty\right\rbrace .
	\end{equation*}
	Denote by $\delta_x(\cdot)$ the Dirac measure at point $x\in\mathbb{R}^d$, which belongs to $\mathcal{P}_{p}(\mathbb{R}^d)$.
	For $p\geq1$, the Wasserstein distance between $\mu,\nu\in\mathcal{P}_{p}(\mathbb{R}^d)$ is defined by
	\begin{equation*}
		\mathbb{W}_p(\mu,\nu):=\inf_{\pi\in\mathcal{C}(\mu,\nu)}\left(  \int_{\mathbb{R}^d\times\mathbb{R}^d}|x-y|^p\pi(dx,dy)\right)^{1/p},
	\end{equation*}
	where  $\mathcal{C}(\mu,\nu)$ is the family of all couplings for $\mu,\nu$, i.e., $\pi(\cdot, \mathbb{R}^d)= \mu(\cdot)$ and $\pi( \mathbb{R}^d,\cdot)= \nu(\cdot)$.

	\begin{rem}\label{rem1}
		For any $\mu\in\mathcal{P}_{p}(\mathbb{R}^d)$ with $p\geq 1$, let $\mathcal{W}_p(\mu)=\left(\int_{\mathbb{R}^d}|x|^p\mu(dx)\right)^{1/p}$. Then, the assertion $	\mathbb{W}_2(\mu,\delta_{0})=	\mathcal{W}_2(\mu)$ holds with $\mu\in\mathcal{P}_{2}(\mathbb{R}^d)$, which is proved in Lemma 2.3 in  \cite{3}.
	\end{rem}
	Consider the following general SMVE of the form:
	\begin{equation}\label{MV0}
		dX_{t}=b(X_{t},\mathcal{L}_{X_t})dt+\sigma(X_{t},\mathcal{L}_{X_t})d W_{t},~~~t\geq0,
	\end{equation}
	with the non-zero initial value $X_{0}\in L^{\tilde{p}}_{\mathcal{F}_0}$, where
	$L^{\tilde{p}}_{\mathcal{F}_0}$
	is the collection of all $\mathcal{F}_0$-measurable random variables $X_{0}$ in $\mathbb{R}^d$ with $\mathbb{E}|X_{0}|^{\tilde{p}}<\infty$ for all $\tilde{p}\geq 2$.
	Here, $\mathcal{L}_{X_t}$ is the law of $X_t$, $W(t)$ is an $m$-dimensional Brownian motion on the given probability space, and $b:\mathbb{R}^d\times\mathcal{P}_2(\mathbb{R}^d)\rightarrow\mathbb{R}^d$, $\sigma:\mathbb{R}^d\times\mathcal{P}_2(\mathbb{R}^d)\rightarrow\mathbb{R}^{d\times m}$ are both Borel-measurable. 
	%$b(x,\mu)$ is a continuous function with respect to $x\in \mathbb{R}^d$.
	Assume that $b(x,\mu), \sigma(x,\mu)$  in this paper satisfy the locally Lipschitz  condition w.r.t. $x$. To achieve the target, we need  the following assumptions.
	%Our main results of this paper are based on the following two assumptions. 

	\begin{ass}\label{ass2}
		There exist positive constants $K_1,K_2$ such that
		\begin{equation*}
			2\langle x-y,b(x,\mu)-b(y,\nu)\rangle+\|\sigma(x,\mu)-\sigma(y,\nu)\|^2\leq-K_1|x-y|^2+K_2\mathbb{W}_{2}(\mu,\nu)^2,
		\end{equation*}
		for any $x,y\in\mathbb{R}^d$ and $\mu,\nu\in\mathcal{P}_2(\mathbb{R}^d)$.
	\end{ass}

	\begin{ass}\label{ass1}
		There exist three  constants $a_1,a_2>0$ and $q\geq2$ such that
		\begin{equation*}
			2\langle x,b(x,\mu)\rangle+(q-1)\|\sigma(x,\mu)\|^2\leq-a_1|x|^2+a_2\mathcal{W}_{2}(\mu)^{2},
		\end{equation*}
		for any $x\in\mathbb{R}^d$ and $\mu\in\mathcal{P}_2(\mathbb{R}^d)$.
	\end{ass}
	
	Under these  assumptions, the SMVE (\ref{MV0}) admits a unique strong solution \cite{51}.
	By employing the proof techniques in Theorems 3.1 and 5.2 in \cite{qiaosiam}, together with the results on pathwise and law uniqueness in \cite{dingqiao,1991}, we can easily  show that the solution to SMVE (\ref{MV0}) remains mean-square and almost surely stable even when the  coefficients exhibit superlinear growth with respect to the state variables.
	Therefore, we only state the result in the following theorem and omit the proof.
	\begin{thm}\label{thm000}
		Let Assumptions \ref{ass2} and \ref{ass1} hold with $a_1>a_2$. Then, the solution $X_{t}$ to SMVE (\ref{MV0}) is stable in mean-square and almost sure sense, i.e.,
		$$\lim_{t\rightarrow\infty}\mathbb{E}|X_{t}|^2=0~~~
		\text{and}~~~
		\mathbb{P}\left\{\lim_{t\rightarrow\infty}|X_{t}|=0\right\}=1.$$
	\end{thm}

	It is the most common way to use the interacting particle system to approximate SMVE, whose crucial point is exploiting the  empirical measure to approximate the law $\mathcal{L}_{X_t}$.
	For any $j\in\mathbb{S}_N:=\{1,2,\cdots,N\}$, let $(W^j,X_0^j)$ be independent copies of $(W,X_0)$. Then, the non-interacting particle system w.r.t. (\ref{MV0}) is introduced by
	\begin{equation}\label{nonMV}
		dX_{t}^{j}=b(X_{t}^{j},\mu_{t}^{X^j})dt+\sigma(X_{t}^{j},\mu_{t}^{X^j})d W_{t}^{j},~~~j\in\mathbb{S}_{N},
	\end{equation}
	with the initial value $X_0^j=X_0$, where $\mu_{t}^{X_j}$ is the law of $X^j$ at $t$. It is obvious that $\mu_{t}^{X^j}=\mu_{t}^{X}$ holds for any $j\in\mathbb{S}_{N}$.
	The corresponding interacting particles system is defined by
	\begin{equation}\label{MVVVVV}
		dX_{t}^{j,N}=b(X_{t}^{j,N},\mu_{t}^{X,N})dt+\sigma(X_{t}^{j,N},\mu_{t}^{X,N})d W_{t}^{j}, ~~~j\in\mathbb{S}_{N},
	\end{equation}
	with the initial value $X_0^j=X_0$, where the empirical measure is defined by $\mu_{t}^{X,N}(\cdot)=\frac{1}{N}\sum_{i=1}^{N}\delta_{X_{t}^{i,N}}(\cdot)$. 
	The mean-square  and almost sure stabilities of the non-interacting particle system (\ref{nonMV}) follow from the corresponding results of SMVEs in \cite{qiaosiam}.
	We will focus on the stability
	of interacting particle system.

	\section{Stability of interacting particle system}

	In this section, the mean-square and almost sure stabilities of the interacting particle system (\ref{MVVVVV}) are presented. Moreover, the  rates of stabilities are also revealed.
	% We first give the definitions of these two  stabilities.
%	\begin{defn}
%		The solution $X_{t}^{j,N}$ to interacting particle system (\ref{MVVVVV}) is said to be  exponentially stable in mean-square sense if 
%		$$\lim_{t\rightarrow\infty}\frac{1}{t}\log(\mathbb{E}|X_{t}^{j,N}|^2)\leq-\lambda_0,$$
%		holds for some positive constant $\lambda_0$ and any $j\in\mathbb{S}_N$.
%	\end{defn}
%	\begin{defn}
%		The solution $X_{t}^{j,N}$ to interacting particle system (\ref{MVVVVV}) is said to be almost surely exponentially stable if 
%		$$\lim_{t\rightarrow\infty}\frac{1}{t}\log(\frac{1}{N}\sum_{i=1}^{N}|X_{t}^{i,N}|^2)\leq-\lambda_1~~~a.s.$$
%		holds for some positive constant $\lambda_1$.
%	\end{defn}
	
	\begin{thm}\label{thm31}
		Let Assumptions \ref{ass2} and \ref{ass1} hold. If $a_1>a_2$, then the solution $X_{t}^{j,N}$ to interacting particle system (\ref{MVVVVV}) is exponentially stable in mean-square sense, i.e., for any $j\in\mathbb{S}_N$ and some $a\in(0,a_1-a_2]$,
		$$\lim_{t\rightarrow\infty}\frac{1}{t}\log(\mathbb{E}|X_{t}^{j,N}|^2)\leq-a.$$
	\end{thm}
	\begin{proof}
		For any $j\in\mathbb{S}_N$ and some $a\in(0,a_1-a_2]$, an application of It\^o's formula to $e^{at}|X_{t}^{j,N}|^2$ gives that
		\begin{equation}\label{31}
			\begin{split}
				e^{at}|X_{t}^{j,N}|^2=&|X_0^j|^2+\int_{0}^{t}ae^{as}|X_{s}^{j,N}|^2ds+2\int_{0}^{t}e^{as}\langle X_{s}^{j,N},b(X_{s}^{j,N},\mu_{s}^{X,N})\rangle ds\\&+\int_{0}^{t}e^{as}\|\sigma(X_{s}^{j,N},\mu_{s}^{X,N})\|^2 ds
				+2\int_{0}^{t}e^{as}\langle X_{s}^{j,N},\sigma(X_{s}^{j,N},\mu_{s}^{X,N})\rangle dW_s^j.
			\end{split}
		\end{equation}
		Since the particles are all  identically distributed for any $j\in\mathbb{S}_{N}$, we derive
		$$\mathbb{E}\big[\frac{1}{N}\sum_{i=1}^{N}|X_{t}^{i,N}|^2\big]=\mathbb{E}\big[|X_{t}^{j,N}|^2\big].$$
		By taking the expectation on both sides of (\ref{31}) and taking Assumption \ref{ass1} into consideration, we have
		\begin{equation*}
			\begin{split}
				e^{at}\mathbb{E}|X_{t}^{j,N}|^2=&\mathbb{E}|X_0^j|^2+\int_{0}^{t}ae^{as}\mathbb{E}|X_{s}^{j,N}|^2ds+2\int_{0}^{t}e^{as}\mathbb{E}\langle X_{s}^{j,N},b(X_{s}^{j,N},\mu_{s}^{X,N})\rangle ds\\&+\int_{0}^{t}e^{as}\mathbb{E}\|\sigma(X_{s}^{j,N},\mu_{s}^{X,N})\|^2 ds\\
				\leq&\mathbb{E}|X_0^j|^2+\int_{0}^{t}e^{as}\big((a-a_1)\mathbb{E}| X_{s}^{j,N}|^2+a_2\mathbb{E}\big[\frac{1}{N}\sum_{i=1}^{N}|X_{s}^{i,N}|^2\big]\big) ds\\
				\leq&\mathbb{E}|X_0^j|^2+\int_{0}^{t}e^{as}(a_2-a_1+a)\mathbb{E}| X_{s}^{j,N}|^2 ds,
			\end{split}
		\end{equation*}
		where we use a result of Wasserstein metric
		$\mathbb{W}_{2}(\mu_{t}^{X,N},\delta_{0})^2\leq\frac{1}{N}\sum_{i=1}^{N}|X_{t}^{i,N}|^2.$
		This implies
		\begin{equation*}
			\mathbb{E}|X_{t}^{j,N}|^2\leq e^{-at}\mathbb{E}|X_0^j|^2.
		\end{equation*}
	%	which gives the desired result.
	\end{proof}
	
	\begin{thm}
		Let Assumptions \ref{ass2} and \ref{ass1} hold. If $a_1>a_2$, then the solution $X_{t}^{j,N}$ to interacting particle system (\ref{MVVVVV}) is almost surely stable, i.e., for some $a\in(0,a_1-a_2]$,
		$$\lim_{t\rightarrow\infty}\frac{1}{t}\log(\frac{1}{N}\sum_{i=1}^{N}|X_{t}^{i,N}|^2)\leq-a~~~a.s.$$
	\end{thm}
	\begin{proof}
		For some $a\in(0,a_1-a_2]$, summing (\ref{31}) from $j=1$ to $N$ and then dividing it by $N$ lead to
		%	By summing over $j=1,\cdots,N$ and*** dividing by $N$***, *(\ref{31})* becomes into
		\begin{equation*}
			\begin{split}
				&	e^{at}\frac{1}{N}\sum_{i=1}^{N}|X_{t}^{i,N}|^2\\=&\frac{1}{N}\sum_{i=1}^{N}|X_0^i|^2+\int_{0}^{t}e^{as}\big((a-a_1)\frac{1}{N}\sum_{i=1}^{N}|X_{s}^{i,N}|^2+a_2\frac{1}{N}\sum_{i=1}^{N}|X_{s}^{i,N}|^2\big)ds\\&
				+2\int_{0}^{t}e^{as}\frac{1}{N}\sum_{i=1}^{N}\langle X_{s}^{i,N},\sigma(X_{s}^{i,N},\mu_{s}^{X,N})\rangle dW_s^i\\
				\leq&\frac{1}{N}\sum_{i=1}^{N}|X_0^i|^2
				+2\int_{0}^{t}e^{as}\frac{1}{N}\sum_{i=1}^{N}\langle X_{s}^{i,N},\sigma(X_{s}^{i,N},\mu_{s}^{X,N})\rangle dW_s^i,
			\end{split}
		\end{equation*}
		where Assumption \ref{ass1} is used.
		Since the last term  above  is a martingale, 
		%\begin{equation*}
		%	\begin{split}
			%		\mathbb{E}\Big[\int_{0}^{t}e^{as}\frac{1}{N}\sum_{i=1}^{N}\langle X_{s}^{i,N},\sigma(X_{s}^{i,N},\mu_{s}^{X,N})\rangle dW_s^i\big|\mathcal{F}_s\Big]=\int_{0}^{s}e^{ar}\frac{1}{N}\sum_{i=1}^{N}\langle X_{r}^{j,N},\sigma(X_{r}^{j,N},\mu_{r}^{X,N})\rangle dW_r^j,
			%	\end{split}
		%\end{equation*}	
		applying Lemma 2.2 in \cite{mao1997} gives  that 
		\begin{equation}\label{con}
			\lim_{t\rightarrow\infty}e^{at}\frac{1}{N}\sum_{i=1}^{N}|X_{t}^{i,N}|^2<\Lambda(\omega)<\infty~~~~a.s.
		\end{equation}
		Then, the assertion holds.
	\end{proof}
	
	\section{Propagation of chaos in infinite horizon}
	
	To reproduce the stabilities of SMVE (\ref{MV0}) by  the numerical scheme, the long-time convergences  from interacting particle system to non-interacting particle system are needed.
	We first give the propagation of chaos in mean-square sense in the following theorem, where Theorem 1 in \cite{fougui} plays a key role.
	\begin{thm}\label{963a}
		Let Assumptions \ref{ass2} and \ref{ass1}  hold. If $2K_2<K_1$ and $a_2<a_1$, then for some $a\in\left(0,({a_1-a_2})\wedge({K_1-2K_2})\right)$, 
		\begin{equation}\label{ress1}
			\frac{1}{N}\sum_{i=1}^{N}\mathbb{E}| X_{t}^{i}-X_{t}^{i,N}|^2\leq\frac{2K_2
				G_{d,q}(\mathbb{E}|X_0|^{q})^{\frac{2}{q}}e^{-at}\Phi(N)}{K_1-2K_2-a},
		\end{equation}
		where $\Phi(N):=\begin{cases}N^{-\frac{1}{2}}+N^{-\frac{q-2}{q}},& d<4 ~\text{and}~q\neq 4,\\N^{-\frac{1}{2}}\log(1+N)+N^{-\frac{q-2}{q}},&d=4~ \text{and}~ q\neq 4,\\N^{-\frac{2}{d}}+N^{-\frac{q-2}{q}},&d>4 ~\text{and}~ q\neq \frac{d}{d-2}.\end{cases}$\\
		This, obviously, implies that for any $t>0$,
		\begin{equation*}
			\lim_{N\rightarrow\infty}\frac{1}{N}\sum_{i=1}^{N}\mathbb{E}| X_{t}^{i}-X_{t}^{i,N}|^2=0,
		\end{equation*}
		and for a fixed $N$,
		\begin{equation*}
			\lim_{t\rightarrow\infty}\frac{1}{N}\sum_{i=1}^{N}\mathbb{E}| X_{t}^{i}-X_{t}^{i,N}|^2=0.
		\end{equation*}
	\end{thm}
	
	\begin{proof}
		
		By It\^o's formula, we derive from (\ref{nonMV}) and (\ref{MVVVVV}) that
		%\begin{equation}\label{uni}
		%	\begin{split}
			%		e^{at}|X_{t}^{j}-X_{t}^{j,N}|^2=&	|X_{0}^{j}-X_{0}^{j,N}|^2+\int_{0}^{t}ae^{as}|X_{s}^{j}-X_{s}^{j,N}|^2ds\\&+2\int_{0}^{t}e^{as}\langle X_{s}^{j}-X_{s}^{j,N},b(X_{s}^{j},\mu_{s}^{X_j})-b(X_{s}^{j,N},\mu_{s}^{X,N})\rangle ds\\
			%		&+\int_{0}^{t}e^{as}\|\sigma(X_{s}^{j},\mu_{s}^{X_j})-\sigma(X_{s}^{j,N},\mu_{s}^{X,N})\|ds\\&
			%		+2\int_{0}^{t}e^{as}\langle X_{s}^{j}-X_{s}^{j,N},\sigma(X_{s}^{j},\mu_{s}^{X_j})-\sigma(X_{s}^{j,N},\mu_{s}^{X,N})\rangle dW_s^j.
			%	\end{split}
		%\end{equation}
		\begin{equation}\label{uni}
			\begin{split}
				|X_{t}^{j}-X_{t}^{j,N}|^2=&	2\int_{0}^{t}\langle X_{s}^{j}-X_{s}^{j,N},b(X_{s}^{j},\mu_{s}^{X^j})-b(X_{s}^{j,N},\mu_{s}^{X,N})\rangle ds\\
				&+\int_{0}^{t}\|\sigma(X_{s}^{j},\mu_{s}^{X^j})-\sigma(X_{s}^{j,N},\mu_{s}^{X,N})\|^2ds\\&
				+2\int_{0}^{t}\langle X_{s}^{j}-X_{s}^{j,N},\sigma(X_{s}^{j},\mu_{s}^{X^j})-\sigma(X_{s}^{j,N},\mu_{s}^{X,N})\rangle dW_s^j.
			\end{split}
		\end{equation}
		Taking the expectation on both sides of (\ref{uni}) and differentiating w.r.t. $t$ yield that
		\begin{equation*}\label{unii}
			\begin{split}
				\frac{d}{dt}\mathbb{E}|X_{t}^{j}-X_{t}^{j,N}|^2=&	2\mathbb{E}\langle X_{t}^{j}-X_{t}^{j,N},b(X_{t}^{j},\mu_{t}^{X^j})-b(X_{t}^{j,N},\mu_{t}^{X,N})\rangle \\
				&+\mathbb{E}\|\sigma(X_{t}^{j},\mu_{t}^{X^j})-\sigma(X_{t}^{j,N},\mu_{t}^{X,N})\|^2.
			\end{split}
		\end{equation*}
		Then, Assumption \ref{ass2} implies
		\begin{equation}\label{der}
			\begin{split}
				\frac{d}{dt}\mathbb{E}|X_{t}^{j}-X_{t}^{j,N}|^2\leq&	-K_1\mathbb{E}| X_{t}^{j}-X_{t}^{j,N}|^2+K_2\mathbb{E}\big[\mathbb{W}_2(\mu_{t}^{X^j},\mu_{t}^{X,N})^2 \big].
			\end{split}
		\end{equation}
		By summing (\ref{der}) from $j=1$ to $N$ and then dividing it by $N$, one can see that
		\begin{equation}\label{zon1}
			\begin{split}
				&\frac{d}{dt}\frac{1}{N}\sum_{i=1}^{N}\mathbb{E}|X_{t}^{i}-X_{t}^{i,N}|^2\\\leq&	-\frac{K_1}{N}\sum_{i=1}^{N}\mathbb{E}| X_{t}^{i}-X_{t}^{i,N}|^2+	\frac{2K_2}{N}\sum_{i=1}^{N}\mathbb{E}|X_{t}^{i}-X_{t}^{i,N}|^2\\
				&+2K_2\mathbb{E}\big[\mathbb{W}_2(\mu_{t}^{X^j},\mu_{t}^{N})^2\big]\\
				\leq&(2K_2-K_1)\frac{1}{N}\sum_{i=1}^{N}\mathbb{E}| X_{t}^{i}-X_{t}^{i,N}|^2+2K_2\mathbb{E}\big[\mathbb{W}_2(\mu_{t}^{X^j},\mu_{t}^{N})^2\big],
			\end{split}
		\end{equation}
		%\begin{equation*}
		%	\begin{split}
			%		\frac{1}{N}\sum_{j=1}^{N}\mathbb{E}|X_{t}^{j}-X_{t}^{j,N}|^2\leq	\int_{0}^{t}-K_1	\frac{1}{N}\sum_{j=1}^{N}\mathbb{E}| X_{s}^{j}-X_{s}^{j,N}|^2+K_2\mathbb{E}\big[\mathcal{W}_2^2(\mu_{s}^{X_j},\mu_{s}^{X,N}) \big]ds.
			%	\end{split}
		%\end{equation*}
		%\begin{equation}\label{gron}
		%	\begin{split}
			%		&e^{at}\frac{1}{N}\sum_{j=1}^{N}\mathbb{E}|X_{t}^{j}-X_{t}^{j,N}|^2\\\leq&\int_{0}^{t}ae^{as}\mathbb{E}|X_{s}^{j}-X_{s}^{j,N}|^2ds+	\int_{0}^{t}-K_1e^{as}	\frac{1}{N}\sum_{j=1}^{N}\mathbb{E}| X_{s}^{j}-X_{s}^{j,N}|^2ds\\&+\int_{0}^{t}2e^{as}K_2	\frac{1}{N}\sum_{j=1}^{N}\mathbb{E}|X_{s}^{j}-X_{s}^{j,N}|^2ds+\int_{0}^{t}2K_2e^{as}\mathbb{E}\big[\mathcal{W}_2^2(\mu_{s}^{X_j},\mu_{s}^{N})\big]ds\\
			%		\leq&\int_{0}^{t}2K_2\mathbb{E}\big[\mathcal{W}_2^2(\mu_{s}^{X_j},\mu_{s}^{N})\big]ds.
			%	\end{split}
		%\end{equation}
		where we use the fact that
		\begin{equation*}
			\begin{split}
				\mathbb{E}\big[\mathbb{W}_2(\mu_{t}^{X^j},\mu_{t}^{X,N})^2\big] \leq2\mathbb{E}\big[\mathbb{W}_2(\mu_{t}^{X^j},\mu_{t}^{N})^2\big]+2\mathbb{E}\big[\mathbb{W}_2(\mu_{t}^{N},\mu_{t}^{X,N})^2\big],
			\end{split}
		\end{equation*}
		with
		$
		\mu_{t}^{N}:=\frac{1}{N}\sum_{i=1}^{N}\delta_{X_{t}^{i}}
		$. 
		By choosing $p=2$ in Theorem 1 of \cite{fougui}, we get that, for $t\geq0$ and some  $q>2$, 
		%	From [qiao], we know that $X_t^j$ is *** bounded uniform in time*** under the given conditions, namely $\sup_{t\geq0}\mathbb{E}|X_t^j|^p<\infty$. ***Moreover, Lemma 4.1 in [Wu hao] exihibits that the constant $C_p$ in (\ref{ppc}) is relavant to the p-th moment $\mathcal{W}_{p}(\mu )$ of i.i.d. random variables, which can be further explained into $C_p$ is dependent on $\sup_{t\geq0}\mathbb{E}|X_t^j|^p$.***
		%	
		%	For $t\geq0$ 
		\begin{equation}\label{ppc}
			\begin{split}
				\mathbb{E}\big[\mathbb{W}_2(\mu_{t}^{X^j},\mu_{t}^{N})^2\big]\leq G_{d,q}(\mathbb{E}|X_t|^{q})^{\frac{2}{q}}\Phi(N),
			\end{split}
		\end{equation}
		where $G_{d,q}$ is a constant which only depends on ${d,q}$, and $$\Phi(N):=\begin{cases}N^{-\frac{1}{2}}+N^{-\frac{q-2}{q}},& d<4 ~\text{and}~q\neq 4,\\N^{-\frac{1}{2}}\log(1+N)+N^{-\frac{q-2}{q}},&d=4~ \text{and}~ q\neq 4,\\N^{-\frac{2}{d}}+N^{-\frac{q-2}{q}},&d>4 ~\text{and}~ q\neq \frac{d}{d-2}.\end{cases}$$
		%	From Theorem 3.1 in [qiao], we know that $X_t^j$ is exponentially stable under the given conditions, which means that by choosing $\alpha=a$, we have
		%	\begin{equation}\label{sspc}
			%		(\mathbb{E}|X_t^j|^{q})^{\frac{2}{q}}<e^{-at}(\mathbb{E}|X_0^j|^{q})^{\frac{2}{q}}
			%	\end{equation}
		Then, we can show 
		\begin{equation}\label{sspc}
			(\mathbb{E}|X_t|^{q})^{\frac{2}{q}}\leq  e^{-at}(\mathbb{E}|X_0|^{q})^{\frac{2}{q}}
		\end{equation}
		Actually, for $\beta>0$, the It\^o formula with Assumption \ref{ass1} leads to
		\begin{equation*}
			\begin{split}
				&e^{\beta t}\mathbb{E}|X_{t}^{j}|^q\\\leq&\mathbb{E}|X_0^j|^q+\int_{0}^{t}\beta e^{\beta s}\mathbb{E}|X_{s}^{j}|^qds \\&+\frac{q}{2}\mathbb{E}\Big[\int_{0}^{t}e^{\beta s}|X_{s}^{j}|^{q-2}\Big(2\langle X_{s}^{j},b(X_{s}^{j},\mu_{s}^{X^j})\rangle+(q-1)\|\sigma(X_{s}^{j},\mu_{s}^{X^j})\|^2\Big) \Big]ds\\
				\leq&\mathbb{E}|X_0^j|^q+\int_{0}^{t}\beta e^{\beta s}\mathbb{E}|X_{s}^{j}|^qds\\
				&+\frac{q}{2}\mathbb{E}\Big[\int_{0}^{t}e^{\beta s}|X_{s}^{j}|^{q-2}\Big(-a_1| X_{s}^{j}|^2+a_2\mathcal{W}_2(\mu_{s}^{X^j})^2\Big) \Big]ds
				\\\leq&\mathbb{E}|X_0^j|^q+\int_{0}^{t}\beta e^{\beta s}\mathbb{E}|X_{s}^{j}|^qds+\frac{q}{2}\int_{0}^{t}e^{\beta s}\Big(-a_1\mathbb{E}| X_{s}^{j}|^q\\&+\frac{a_2(q-2)}{q}\mathbb{E}| X_{s}^{j}|^q+\frac{2a_2}{q}\mathbb{E}\big[\mathcal{W}_2(\mu_{s}^{X^j})^q\big]\Big)ds\\
				\leq&\mathbb{E}|X_0^j|^q+\int_{0}^{t} e^{\beta s}\big(\beta-\frac{qa_1}{2}+\frac{qa_2}{2}\big)  \mathbb{E}| X_{s}^{j}|^qds.\\
				%\leq&\mathbb{E}|X_0^j|^q.
			\end{split}
		\end{equation*}
		Then, (\ref{sspc}) holds by letting $\beta=\frac{qa}{2}$.
		Therefore, inserting (\ref{ppc}) and (\ref{sspc}) into (\ref{zon1}) yields that
		\begin{equation}\label{222}
			\begin{split}
				&\frac{d}{dt}\frac{1}{N}\sum_{i=1}^{N}\mathbb{E}|X_{t}^{i}-X_{t}^{i,N}|^2\\
				&\leq  (2K_2-K_1)\frac{1}{N}\sum_{i=1}^{N}\mathbb{E}| X_{t}^{i}-X_{t}^{i,N}|^2+2K_2
				G_{d,q}(\mathbb{E}|X_0|^{q})^{\frac{2}{q}}e^{-at}\Phi(N).
			\end{split}
		\end{equation}
		Denote
		$$\varrho(t)=\frac{1}{N}\sum_{i=1}^{N}\mathbb{E}| X_{t}^{i}-X_{t}^{i,N}|^2.$$
		From (\ref{222}), it holds that 
		\begin{equation*}
			\varrho'(t)\leq(2K_2-K_1)\varrho(t)+2K_2
			G_{d,q}(\mathbb{E}|X_0|^{q})^{\frac{2}{q}}e^{-at}\Phi(N),
		\end{equation*}
		with $\varrho(0)=0$.
		By integrating this Gronwall-like differential inequality, we draw the conclusion that, for $t>0$, 
		%	\begin{equation*}
			%		\lim_{t\rightarrow\infty}\lim_{N\rightarrow\infty}\frac{1}{N}\sum_{i=1}^{N}\mathbb{E}| X_{t}^{i}-X_{t}^{i,N}|^2=\lim_{t\rightarrow\infty}\lim_{N\rightarrow\infty}\mathbb{E}| X_{t}^{j}-X_{t}^{j,N}|^2=0.
			%	\end{equation*}
		\begin{equation*}
			\frac{1}{N}\sum_{i=1}^{N}\mathbb{E}| X_{t}^{i}-X_{t}^{i,N}|^2\leq\frac{2K_2
				G_{d,q}(\mathbb{E}|X_0|^{q})^{\frac{2}{q}}e^{-at}\Phi(N)}{K_1-2K_2-a}.
		\end{equation*}
		So the result is proved.
	\end{proof}

	Based on  Chebyshev's inequality and Borel-Cantelli's lemma with  (\ref{ress1}), the following theorem is concerned with the almost sure convergence in infinite horizon of  interacting particle system to non-interacting particle system.
	\begin{thm}
		Let all conditions in Theorem \ref{963a} hold. Then, 
		\begin{equation*}
			\lim_{t\rightarrow\infty}\lim_{N\rightarrow\infty}\frac{1}{N}\sum_{i=1}^{N}| X_{t}^{i}-X_{t}^{i,N}|^2=0~~~~a.s.
		\end{equation*}
	\end{thm}
	\begin{proof}
		%	Firstly, due to the ***moment*** boundedness of the solutions $X^{j,N}$ and $X^j$ with respect to interacting particle system and non-interacting particle system, respectively, one can see that
		%	\begin{equation*}
			%		\mathbb{E}| X_{s}^{j}-X_{s}^{j,N}|^2\leq2\mathbb{E}| X_{s}^{j}|^2+2\mathbb{E}| X_{s}^{j,N}|^2.
			%	\end{equation*}
		%	
		Firstly, let $\epsilon\in(0,a)$ be arbitrary. By the technique of Chebyshev's inequality and (\ref{ress1}), we have
		\begin{equation*}
			\begin{split}
				&\mathbb{P}\Big(\frac{1}{N}\sum_{i=1}^{N}| X_{t}^{i}-X_{t}^{i,N}|^2>e^{-(a-\epsilon)t}\Phi(N)\Big)\\
				&\leq \frac{\frac{1}{N}\sum_{i=1}^{N}\mathbb{E}| X_{t}^{i}-X_{t}^{i,N}|^2}{e^{-(a-\epsilon)t}\Phi(N)}\\
				&\leq\frac{2K_2
					G_{d,q}(\mathbb{E}|X_0|^{q})^{\frac{2}{q}}e^{-at}\Phi(N)}{(K_1-2K_2-a)e^{-(a-\epsilon)t}\Phi(N)}
				\\
				&\leq \frac{2K_2
					G_{d,q}(\mathbb{E}|X_0|^{q})^{\frac{2}{q}}e^{-\varepsilon t}}{(K_1-2K_2-a)},
			\end{split}
		\end{equation*}
		where $\Phi(N)$ is defined in Theorem \ref{963a}.	The Borel-Cantelli lemma allows us to know that, for almost all $\omega\in\Omega$, there exists a positive integer $N_0=N_0(\omega)$ such that
		% \forall \omega\in \Omega$  %$(\ref{42a})$ holds whenever $t\geq T^*$,
		\begin{equation}\label{42a}
			\frac{1}{N}\sum_{i=1}^{N}| X_{t}^{i}-X_{t}^{i,N}|^2\leq e^{-(a-\epsilon)t}\Phi(N),
		\end{equation}
		whenever $N\geq N_0$. Thus, for any $t>0$,
		%	holds for all but finite $t$. Thus, there exists a $T^*(\omega), \forall \omega\in \Omega$ such that $(\ref{42a})$ holds whenever $t\geq T^*$, which means 
		\begin{equation*}
			\lim_{N\rightarrow\infty}\frac{1}{N}\sum_{i=1}^{N}| X_{t}^{i}-X_{t}^{i,N}|^2=0~~~~a.s.
		\end{equation*}
	\end{proof}
	%**************
	%
	%Furthermore, it is worth noting that EM numerical solution of IPS can reproduce the almost sure stability of SMVE's solution.
	%
	%
	%Define
	%\begin{equation*}
	%	\Omega_1:=\{\omega:i\in\mathbb{S}_r,\text{the solution of IPS is stability for a long time} \},
	%\end{equation*}
	%\begin{equation*}
	%	\Omega_2:=\{\omega:i\in\mathbb{S}_r,\text{the solutions of IPS and non-IPS are convergent uniformly in time} \}.
	%\end{equation*}
	%We have already known that under the given assumptions, $\mathbb{P}(\Omega_1)=\mathbb{P}(\Omega_2)=1$. Therefore,
	%\begin{equation*}
	%	\begin{split}
		%		1\geq\mathbb{P}\big(\Omega_1\cap\Omega_2\big)&= \mathbb{P}\big(\Omega_1\big)+\mathbb{P}\big(\Omega_2\big)-\mathbb{P}\big(\Omega_1\cup\Omega_2\big)\\
		%		&\geq\mathbb{P}\big(\Omega_1\big)+\mathbb{P}\big(\Omega_2\big)-1
		%	\end{split}
	%\end{equation*}
	%which  means $
	%\mathbb{P}\big(\Omega_1\cap\Omega_2\big)=1$.
	%
	%And by taking advantage of () and (), the convergence in mean square can be shown .
	%
	%**********************
	%
	%

	\section{Linear coefficient: stability of EM scheme}
	
	We present the stability analysis of EM scheme for the corresponding particle system in this section. After constructing the conventional EM scheme for interacting particle system, we show the mean-square and almost sure stabilities of the numerical solution.
	
	Suppose that there exists a positive integer $M$ 
	such that $\Delta=\frac{1}{M}$.
	For each $j\in\mathbb{S}_{N}$ and the given time-step $\Delta$, the EM scheme of (\ref{MVVVVV}) in discretization version is:
	\begin{equation}\label{num}
		Y_{t_{k+1}}^{j,N}=Y_{t_{k}}^{j,N}+b(Y_{t_{k}}^{j,N},\mu_{t_{k}}^{Y,N})\Delta+\sigma(Y_{t_{k}}^{j,N},\mu_{t_{k}}^{Y,N})\Delta W_{t_k}^{j},
	\end{equation}
	where $Y_{0}^{j,N}=X_{0}^{j}$, $\Delta W_{t_k}^{j}=W_{t_{k+1}}^{j}-W_{t_k}^{j}$ and $\mu_{t_{k}}^{Y,N}(\cdot)=\frac{1}{N}\sum_{i=1}^{N}\delta_{Y_{t_{k}}^{i,N}}(\cdot)$.
%	The definitions of two kinds  stabilities of EM scheme for interacting particle system (\ref{MVVVVV}) are stated as follows.
%	\begin{defn}
	%	The numerical solution $Y_{t_{k}}^{j,N}$ of EM scheme (\ref{num}) is said to be  exponentially stable in mean-square sense if 
	%	$$\lim_{k\rightarrow\infty}\frac{1}{k\Delta}\log\big(	\mathbb{E}|Y_{t_{k}}^{j,N}|^2\big)\leq-\gamma_0,$$
	%	holds for some positive constant $\gamma_0$ and any $j\in\mathbb{S}_N$.
	%\end{defn}
	%\begin{defn}
		%The numerical solution $Y_{t_{k}}^{j,N}$ of EM scheme (\ref{num}) is said to be almost surely exponentially stable if 
		%$$\lim_{k\rightarrow\infty}\frac{1}{k\Delta}\log	\big(\frac{1}{N}\sum_{i=1}^{N}|Y_{t_{k}}^{i,N}|^2\big)\leq-\gamma_1~~~a.s.$$
		%holds for some positive constant $\gamma_1$.
	%\end{defn}
	
	\subsection{Mean-square stability of  EM scheme}
	To achieve the goal, the condition that the drift coefficients grow linearly is imposed.
	\begin{ass}\label{asas}
		There exist two constants $b_1,b_2>0$ such that
		\begin{equation*}
			|b(x,\mu)|^2\leq b_1|x|^2+b_2\mathcal{W}_{2}(\mu)^2,
		\end{equation*}
		for any $x\in\mathbb{R}^d$ and $\mu\in\mathcal{P}_2(\mathbb{R}^d)$.
	\end{ass}

	\begin{thm}\label{meanE}
		Let Assumptions \ref{ass2}, \ref{ass1},  and \ref{asas} hold with $0< a_2<a_1$. Then, for any $j\in\mathbb{S}_{N}$, 
		there
		exists a stepsize
		$\Delta_0\in (0,1)$ such that
		for any $\Delta\in(0,\Delta_0)$,
		the numerical solution $Y_{t_{k}}^{j,N}$ is  exponentially stable in mean-square sense, i.e.,
		$$\lim_{k\rightarrow\infty}\frac{1}{k\Delta}\log\big(	\mathbb{E}|Y_{t_{k}}^{j,N}|^2\big)\leq -\theta_{\Delta}^{*},$$
		where $\theta_{\Delta}^{*}$ is a positive constant satisfying
		$$\lim_{\Delta\rightarrow0}{\theta_{\Delta}^{*}}=a_1-a_2.$$ Here, the constraint on $\Delta_0$ is given in the proof.
	\end{thm}
	\begin{proof}

		After squaring and taking expectation on both sides of (\ref{num}), we obtain  that
		\begin{equation*}
			\begin{split}
				\mathbb{E}|Y_{t_{k+1}}^{j,N}|^2=&\mathbb{E}|Y_{t_{k}}^{j,N}|^2+\mathbb{E}|b(Y_{t_{k}}^{j,N},\mu_{t_{k}}^{Y,N})|^2\Delta^2+\mathbb{E}\|\sigma(Y_{t_{k}}^{j,N},\mu_{t_{k}}^{Y,N})\|^2\Delta \\
				&+2\mathbb{E}\langle Y_{t_{k}}^{j,N},b(Y_{t_{k}}^{j,N},\mu_{t_{k}}^{Y,N})\rangle\Delta.
			\end{split}
		\end{equation*}
		%\\According to the definition, there is a result of Wasserstein metric:
		%
		%$$\mathcal{W}_{2}(\mu_{t_{k}}^{N},\delta_{0})^2\leq\frac{1}{N}\sum_{j=1}^{N}|Y_{t_k}^{j,N}|^2.$$
		%And, because of the i.i.d of particles, for every $j\in\mathbb{S}_{N}$,
		%$$\mathbb{E}\big[\frac{1}{N}\sum_{j=1}^{N}|Y_{t_k}^{j,N}|^2\big]=\mathbb{E}\big[|Y_{t_k}^{j,N}|^2\big]$$holds.
		%Due to the independent distribution of the particles, 
		By Assumptions \ref{ass1} and \ref{asas}, one can see that
		\begin{equation}\label{ms1}
			\begin{split}
				\mathbb{E}|Y_{t_{k+1}}^{j,N}|^2\leq&\mathbb{E}|Y_{t_{k}}^{j,N}|^2+b_1\Delta^2\mathbb{E}|Y_{t_{k}}^{j,N}|^2+\frac{b_2}{N}\Delta^2\sum_{i=1}^{N}\mathbb{E}|Y_{t_{k}}^{i,N}|^2-a_1\mathbb{E}|Y_{t_{k}}^{j,N}|^2\Delta\\&+\frac{a_2}{N}\Delta\sum_{i=1}^{N}\mathbb{E}|Y_{t_{k}}^{i,N}|^2\\
				\leq&\mathbb{E}|Y_{t_{k}}^{j,N}|^2+(b_1+b_2)\Delta^2\mathbb{E}|Y_{t_{k}}^{j,N}|^2+(a_2-a_1)\Delta\mathbb{E}|Y_{t_{k}}^{j,N}|^2.
			\end{split}
		\end{equation}
		For any constant $\lambda>1$, we get 
		\begin{equation}\label{ms2}
			\begin{split}
				&\lambda^{(k+1)\Delta}	|Y_{t_{k+1}}^{j,N}|^2-\lambda^{k\Delta}	|Y_{t_{k}}^{j,N}|^2\\&=\lambda^{(k+1)\Delta}\big(|Y_{t_{k+1}}^{j,N}|^2-|Y_{t_{k}}^{j,N}|^2\big)+\big(\lambda^{(k+1)\Delta}-\lambda^{k\Delta}\big)|Y_{t_{k}}^{j,N}|^2.
			\end{split}
		\end{equation}
		Taking (\ref{ms1}) and (\ref{ms2}) into consideration leads to
		\begin{equation*}
			\begin{split}
				&\lambda^{(k+1)\Delta}	\mathbb{E}|Y_{t_{k+1}}^{j,N}|^2-\lambda^{k\Delta}\mathbb{E}|Y_{t_{k}}^{j,N}|^2\\
				\leq&\lambda^{(k+1)\Delta}(b_1+b_2)\Delta^2\mathbb{E}|Y_{t_{k}}^{j,N}|^2+\lambda^{(k+1)\Delta}(a_2-a_1)\Delta\mathbb{E}|Y_{t_{k}}^{j,N}|^2\\
				&+\big(\lambda^{(k+1)\Delta}-\lambda^{k\Delta}\big)\mathbb{E}|Y_{t_{k}}^{j,N}|^2\\
				\leq& \lambda^{(k+1)\Delta}\big[(b_1+b_2)\Delta^2+(a_2-a_1)\Delta+1-\lambda^{-\Delta}\big]\mathbb{E}|Y_{t_{k}}^{j,N}|^2,
			\end{split}
		\end{equation*}
		which yields
		%By iterating, we obtain
		\begin{equation*}
			\lambda^{k\Delta}\mathbb{E}|Y_{t_{k}}^{j,N}|^2-\mathbb{E}|Y_{0}^{j,N}|^2\leq \sum_{l=0}^{k-1}\lambda^{(l+1)\Delta}\big[(b_1+b_2)\Delta^2+(a_2-a_1)\Delta+1-\lambda^{-\Delta}\big]\mathbb{E}|Y_{t_{l}}^{j,N}|^2.
		\end{equation*}
		Define 
		$$h(\lambda)=\lambda^{\Delta}\big((b_1+b_2)\Delta^2+(a_2-a_1)\Delta+1\big)-1.$$
		Denote $\Delta_{1}=\frac{a_1-a_2}{b_1+b_2}$. It is obvious that for any $\Delta<\Delta_{1}\wedge 1$, the assertion $h(1)<0$ holds. Moreover, there exists a $\Delta_{2}$ such that for any $\Delta<\Delta_{2}\wedge 1$, the assertion $(b_1+b_2)\Delta^2+(a_2-a_1)\Delta+1>0$ holds, which means that for any $\lambda>1$, $h'(\lambda)>0$ and $\lim_{\lambda\rightarrow\infty}h(\lambda)=\infty$.
		Therefore, for any $\Delta\in(0,\Delta_0)$ with $\Delta_0=\Delta_{1}\wedge\Delta_{2}\wedge 1$, there exists a unique $\lambda_{\Delta}^{*}>1$ such that $h(\lambda_{\Delta}^{*})=0$, where $\lambda_{\Delta}^{*}$ depends on the stepsize $\Delta$. Letting $\lambda=\lambda_{\Delta}^{*}$ gives that
		$$	{\lambda^{*}_{\Delta}}^{k\Delta}\mathbb{E}|Y_{t_{k}}^{j,N}|^2\leq\mathbb{E}|Y_{0}^{j,N}|^2.$$
		By choosing ${\theta_{\Delta}^{*}}>0$ such that $\lambda_{\Delta}^{*}=\exp({{\theta_{\Delta}^{*}}})$, we can get that
		$$	\mathbb{E}|Y_{t_{k}}^{j,N}|^2\leq \exp(-k\Delta{\theta_{\Delta}^{*}} )\mathbb{E}|Y_{0}^{j,N}|^2.$$
		Define 
		$$\tilde{h}(\lambda)=\frac{h(\lambda)}{\Delta \lambda^{\Delta}}=(b_1+b_2)\Delta+a_2-a_1+\frac{(1-\lambda^{-\Delta})}{\Delta},$$
		with $\tilde{h}(\lambda_{\Delta}^{*})=h(\lambda_{\Delta}^{*})=0$. \\
		It is worth noting that $\lim_{\Delta\rightarrow0}\frac{(1-{\lambda_{\Delta}^{*}}^{-\Delta})}{\Delta}=\lim_{\Delta\rightarrow0}{\theta_{\Delta}^{*}}$. 
		Thus,
		\begin{equation*}
			\lim_{\Delta\rightarrow0}\tilde{h}(\lambda_{\Delta}^{*})=\lim_{\Delta\rightarrow0}{\theta_{\Delta}^{*}}+a_2-a_1=0.
		\end{equation*}
		The desired result can be proved.
		%$$\lim_{k\rightarrow\infty}\sup\frac{1}{k\Delta}\log	\mathbb{E}|Y_{t_{k}}^{j,N}|^2\leq a_2-a_1.$$
	\end{proof}
	
	\begin{rem}\label{remean}
		By Theorem \ref{meanE}, we observe that the rate of stability is related to the stepsize. When the stepsize is smaller, the rate is larger. This will be illustrated by Figure \ref{tu0} in numerical simulation.
	\end{rem}
	\subsection{Almost sure stability of  EM scheme }
	
	To derive the almost sure exponential stability of EM scheme, the dissipative condition involving the dimension of Brownian motion is needed.
	\begin{ass}\label{aqsqsE}
		There exist constants $c_1,c_2>0$ such that
		\begin{equation*}
			2\langle x,b(x,\mu)\rangle+m\|\sigma(x,\mu)\|^2\leq-c_1|x|^2+c_2\mathcal{W}_{2}(\mu)^{2},
		\end{equation*}
		for any $x\in\mathbb{R}^d$ and $\mu\in\mathcal{P}_2(\mathbb{R}^d)$, where $m$ is the dimension of Brownian motion.
	\end{ass}
	
	\begin{thm}\label{almost}
		Let Assumptions \ref{ass2},  \ref{asas}, and \ref{aqsqsE}  hold with $0< c_2<c_1$. Then, there
		exists a stepsize
		$\bar{\Delta}_0\in (0,1)$ such that
		for any $\Delta\in(0,\bar{\Delta}_0)$, the numerical solution $Y_{t_{k}}^{j,N}$ is almost surely exponentially stable, i.e., 
		$$\lim_{k\rightarrow\infty}\frac{1}{k\Delta}\log	\big(\frac{1}{N}\sum_{i=1}^{N}|Y_{t_{k}}^{i,N}|^2\big)\leq -{\xi_{\Delta}^{*}},$$
		where ${\xi_{\Delta}^{*}}$ is a positive constant satisfying 
		\begin{equation*}
			\lim_{\Delta\rightarrow0}	{\xi_{\Delta}^{*}}=c_1-c_2.
		\end{equation*}
		Here, we provide the constraint on $\bar{\Delta}_0$  in the proof.
	\end{thm}

	\begin{proof}
		
		Firstly, we transform (\ref{num}) into
		%\begin{equation*}
		%	\begin{split}
			%	|Y_{t_{k+1}}^{j,N}|^2=&|Y_{t_{k}}^{j,N}|^2+|b(Y_{t_{k}}^{j,N},\mu_{t_{k}}^{N})|^2\Delta^2+\|\sigma(Y_{t_{k}}^{j,N},\mu_{t_{k}}^{N})\|^2|\Delta W_{t_k}^{j}| \\
			%		&+2\langle Y_{t_{k}}^{j,N},b(Y_{t_{k}}^{j,N},\mu_{t_{k}}^{N})\rangle\Delta+2\langle Y_{t_{k}}^{j,N},\sigma(Y_{t_{k}}^{j,N},\mu_{t_{k}}^{N})\Delta W_{t_k}^{j}\rangle\\
			%		&+2\langle b(Y_{t_{k}}^{j,N},\mu_{t_{k}}^{N})\Delta,\sigma(Y_{t_{k}}^{j,N},\mu_{t_{k}}^{N})\Delta W_{t_k}^{j}\rangle
			%	\end{split}
		%\end{equation*}
		
		\begin{equation*}
			\begin{split}
				|Y_{t_{k+1}}^{j,N}|^2=&|Y_{t_{k}}^{j,N}|^2+\Delta\big(2\langle Y_{t_{k}}^{j,N},b(Y_{t_{k}}^{j,N},\mu_{t_{k}}^{Y,N})\rangle+m\|\sigma(Y_{t_{k}}^{j,N},\mu_{t_{k}}^{Y,N})\|^2 \\&+|b(Y_{t_{k}}^{j,N},\mu_{t_{k}}^{Y,N})|^2\Delta
				\big)+\gamma_{k}^{j},
			\end{split}
		\end{equation*}
		where 
		\begin{equation*}
			\begin{split}
				\gamma_{k}^{j}=&	\|\sigma(Y_{t_{k}}^{j,N},\mu_{t_{k}}^{Y,N})\|^2|\Delta W_{t_k}^{j}|^2 -m	\|\sigma(Y_{t_{k}}^{j,N},\mu_{t_{k}}^{Y,N})\|^2\Delta\\&+2\langle Y_{t_{k}}^{j,N},\sigma(Y_{t_{k}}^{j,N},\mu_{t_{k}}^{Y,N})\Delta W_{t_k}^{j}\rangle
				+2\langle b(Y_{t_{k}}^{j,N},\mu_{t_{k}}^{Y,N})\Delta,\sigma(Y_{t_{k}}^{j,N},\mu_{t_{k}}^{Y,N})\Delta W_{t_k}^{j}\rangle.
			\end{split}
		\end{equation*}
		By Assumptions \ref{asas} and \ref{aqsqsE}, we obtain
		%\begin{equation*}
		%	\begin{split}
			%		|Y_{t_{k+1}}^{j,N}|^2\leq&|Y_{t_{k}}^{j,N}|^2+\Delta\big(-a_1|Y_{t_k}|^2+a_2\mathbb{W}_{2}^{2}(\mu_{t_{k}}^{N})+b_1\Delta|Y_{t_k}|^2+b_2\Delta\mathbb{W}_{2}^{2}(\mu_{t_{k}}^{N})	\big)+m_{k}\\
			%		\leq&|Y_{t_{k}}^{j,N}|^2+\Delta(b_1\Delta-a_1)|Y_{t_k}|^2+\Delta(b_2\Delta+a_2)\mathbb{W}_{2}^{2}(\mu_{t_{k}}^{N})+m_{k}^{j}.
			%	\end{split}
		%\end{equation*}
		%Let $s_1=\Delta(b_1\Delta^2-a_1\Delta)$ and $s_2=\Delta(b_2\Delta^2+a_2\Delta)$. 
		\begin{equation}\label{as1}
			\begin{split}
				|Y_{t_{k+1}}^{j,N}|^2\leq&|Y_{t_{k}}^{j,N}|^2+\Delta\big(-c_1|Y_{t_k}^{j,N}|^2+\frac{c_2}{N}\sum_{i=1}^{N}|Y_{t_k}^{i,N}|^2+b_1\Delta|Y_{t_k}^{j,N}|^2\\
				&+\frac{b_2\Delta}{N}\sum_{i=1}^{N}|Y_{t_k}^{i,N}|^2\big)+\gamma_{k}^{j}\\
				\leq&|Y_{t_{k}}^{j,N}|^2+\Delta(b_1\Delta-c_1)|Y_{t_k}^{j,N}|^2+\Delta(b_2\Delta+c_2)\frac{1}{N}\sum_{i=1}^{N}|Y_{t_k}^{i,N}|^2+\gamma_{k}^{j}.
			\end{split}
		\end{equation}
		Combining (\ref{ms2}) and (\ref{as1}) gives that
		\begin{equation*}
			\begin{split}
				&\lambda^{(k+1)\Delta}	|Y_{t_{k+1}}^{j,N}|^2-\lambda^{k\Delta}	|Y_{t_{k}}^{j,N}|^2\\
				\leq& \lambda^{(k+1)\Delta}(b_1\Delta^2-c_1\Delta)|Y_{t_{k}}^{j,N}|^2+\lambda^{(k+1)\Delta}(b_2\Delta^2+c_2\Delta)\frac{1}{N}\sum_{i=1}^{N}|Y_{t_k}^{i,N}|^2\\&+\lambda^{(k+1)\Delta}\gamma_{k}^{j}+\big(\lambda^{(k+1)\Delta}-\lambda^{k\Delta}\big)|Y_{t_{k}}^{j,N}|^2\\
				\leq& \lambda^{(k+1)\Delta}\big[(1+b_1\Delta^2-c_1\Delta-\lambda^{-\Delta})|Y_{t_{k}}^{j,N}|^2+(b_2\Delta^2+c_2\Delta)\frac{1}{N}\sum_{i=1}^{N}|Y_{t_k}^{i,N}|^2+\gamma_{k}^{j}\big],
			\end{split}
		\end{equation*}
		which implies that
		\begin{equation*}
			\begin{split}
				&\lambda^{k\Delta}	|Y_{t_{k}}^{j,N}|^2-|Y_{0}^{j,N}|^2\\
				\leq& \sum_{l=0}^{k-1}\lambda^{(l+1)\Delta}\big[(1+b_1\Delta^2-c_1\Delta-\lambda^{-\Delta})|Y_{t_{l}}^{j,N}|^2\\
				&+(b_2\Delta^2+c_2\Delta)\frac{1}{N}\sum_{i=1}^{N}|Y_{t_l}^{i,N}|^2+\gamma_{l}^{j}\big].
			\end{split}
		\end{equation*}
		Let us introduce the function
		\begin{equation}\label{f}
			f(\lambda)=\lambda^{\Delta}(1+b_1\Delta^2-c_1\Delta)-1.
		\end{equation}
		Choose $\bar{\Delta}_{1}>0$ to make $1+b_1\Delta^2-c_1\Delta>0$ hold for any $\Delta<\bar{\Delta}_{1}\wedge 1$. Then, we derive that $f'(\lambda)>0$ and $\lim_{\lambda\rightarrow{\infty}}f(\lambda)={\infty}$ for any $\lambda>1$. Moreover, there has $\bar{\Delta}_{2}$ such that for any $\Delta<\bar{\Delta}_{2}\wedge 1$, $f(1)<0$.
		Therefore, for any $\Delta<\bar{\Delta}_{1}\wedge\bar{\Delta}_{2}\wedge 1$, there exists a unique ${{\vartheta_{\Delta}^{*}}}>1$ such that $f({{\vartheta_{\Delta}^{*}}})=0$, where ${{{\vartheta_{\Delta}^{*}}}}$ depends on the stepsize $\Delta$. By taking $\lambda={{\vartheta_{\Delta}^{*}}}$, we have
		\begin{equation}\label{jN}
			\begin{split}
				{{\vartheta_{\Delta}^{*}}}^{k\Delta}	|Y_{t_{k}}^{j,N}|^2
				\leq &|Y_{0}^{j,N}|^2+\frac{1}{N}(b_2\Delta^2+c_2\Delta)\sum_{l=0}^{k-1}\sum_{i=1}^{N}	{{\vartheta_{\Delta}^{*}}}^{(l+1)\Delta}|Y_{t_{l}}^{i,N}|^2\\
				&+\sum_{l=0}^{k-1}	{{\vartheta_{\Delta}^{*}}}^{(l+1)\Delta}\gamma_{l}^{j}.
			\end{split}
		\end{equation}
		Every particle satisfies (\ref{jN}). Summing $N$ inequalities gives the result
		\begin{equation*}
			\begin{split}
				\sum_{i=1}^{N}	{{\vartheta_{\Delta}^{*}}}^{k\Delta}	|Y_{t_{k}}^{i,N}|^2
				\leq& \sum_{i=1}^{N}|Y_{0}^{i,N}|^2+(b_2\Delta^2+c_2\Delta)\sum_{l=0}^{k-1}\sum_{i=1}^{N}	{{\vartheta_{\Delta}^{*}}}^{(l+1)\Delta}|Y_{t_{l}}^{i,N}|^2\\&+\sum_{i=1}^{N}\sum_{l=0}^{k-1}{{\vartheta_{\Delta}^{*}}}^{(l+1)\Delta}\gamma_{l}^{i}.
			\end{split}
		\end{equation*}
		Let $A_l=	\sum_{i=1}^{N}{{\vartheta_{\Delta}^{*}}}^{l\Delta}	|Y_{t_{l}}^{i,N}|^2$, $B_l=\sum_{i=1}^{N}{{\vartheta_{\Delta}^{*}}}^{(l+1)\Delta}\gamma_{l}^{i}$, $l\in\{0,1,\cdots,k-1\}$, and 
		\begin{equation*}
			\begin{split}
				\Upsilon_k
				= \sum_{i=1}^{N}|Y_{0}^{i,N}|^2+{\vartheta_{\Delta}^{*}}^\Delta(b_2\Delta^2+c_2\Delta)\sum_{l=0}^{k-1}A_l+\sum_{l=0}^{k-1}B_l.
			\end{split}
		\end{equation*}
		We construct
		\begin{equation}\label{Xx}
			\begin{split}
				\Upsilon_k-\Upsilon_{k-1}
				= {{\vartheta_{\Delta}^{*}}}^\Delta(b_2\Delta^2+c_2\Delta)A_{k-1}+B_{k-1}.
			\end{split}
		\end{equation}
		Since $A_l\leq \Upsilon_l$ for $l\in\{0,1,\cdots,k\}$, (\ref{Xx}) becomes into
		\begin{equation}\label{inter}
			\begin{split}
				\Upsilon_k=&\Upsilon_{k-1}
				+ {\vartheta_{\Delta}^{*}}^\Delta(b_2\Delta^2+c_2\Delta)A_{k-1}+B_{k-1}\\
				\leq& \big(1+{\vartheta_{\Delta}^{*}}^\Delta(b_2\Delta^2+c_2\Delta)\big)\Upsilon_{k-1}+B_{k-1}.
			\end{split}
		\end{equation}
		Owing to (\ref{inter}), it can be obtained by iteration that
		\begin{equation*}
			\begin{split}
				\Upsilon_k
				\leq \big(1+{\vartheta_{\Delta}^{*}}^\Delta(b_2\Delta^2+c_2\Delta)\big)^{k-1}\Upsilon_{1}+\sum_{l=1}^{k-1}\big(1+{\vartheta_{\Delta}^{*}}^\Delta(b_2\Delta^2+c_2\Delta)\big)^{k-1-l}B_{l}.
			\end{split}
		\end{equation*}
		Thus,
		\begin{equation}\label{eee}
			\begin{split}
			&	\sum_{i=1}^{N}{\vartheta_{\Delta}^{*}}^{k\Delta}	|Y_{t_{k}}^{i,N}|^2\\
				\leq& \exp\big({k\Delta {\vartheta_{\Delta}^{*}}^\Delta(b_2\Delta+c_2)}\big)	\Big(\sum_{i=1}^{N}|Y_{0}^{i,N}|^2+\sum_{l=1}^{k-1}\frac{B_{l}}{\big(1+{\vartheta_{\Delta}^{*}}^\Delta(b_2\Delta^2+c_2\Delta)\big)^{l}}\Big),
			\end{split}
		\end{equation}
		which follows from an elementary inequality $1+z\leq e^z$ for $z\in\mathbb{R}$.
		
		%
		%	\\&+\exp\big({(k-1)\Delta {\vartheta_{\Delta}^{*}}^\Delta(b_2\Delta+c_2)}\big)\sum_{l=1}^{k-1}B_{l},
		
		%Then, by choosing $\Delta_3$ such that for $\Delta<\Delta_3$, ${\vartheta_{\Delta}^{*}}^\Delta(b_2\Delta+a_2)\leq1$, and dividing both sides of (\ref{eee}) by $e^{k\Delta}$, we obtain that
		
		Then, dividing both sides of (\ref{eee}) by $N \exp\big({k\Delta {\vartheta_{\Delta}^{*}}^\Delta(b_2\Delta+c_2)}\big)$ yields that
		
		\begin{equation}\label{sure}
			\begin{split}
				&	(\frac{{\vartheta_{\Delta}^{*}}^{k\Delta}}{ \exp\big({k\Delta {\vartheta_{\Delta}^{*}}^\Delta(b_2\Delta+c_2)}\big)})\frac{1}{N}\sum_{i=1}^{N}	|Y_{t_{k}}^{i,N}|^2
				\\\leq	&\frac{1}{N}\sum_{i=1}^{N} |Y_{0}^{i,N}|^2+\frac{1}{N}\sum_{l=1}^{k-1}\frac{B_{l}}{\big(1+{\vartheta_{\Delta}^{*}}^\Delta(b_2\Delta^2+c_2\Delta)\big)^{l}}.
			\end{split}
		\end{equation}
		%\begin{equation}
		%	\begin{split}
			%		(\frac{\vartheta_{\Delta}^{*}}{e})^{k\Delta}\sum_{j=1}^{N}	|Y_{t_{k}}^{j,N}|^2
			%		\leq e^{-\Delta}N|Y_{0}^{j,N}|^2+e^{-\Delta}\sum_{l=1}^{k-1}B_{l}.
			%	\end{split}
		%\end{equation}
		Define 
		\begin{equation*}
			H_k=	\frac{1}{N}\sum_{i=1}^{N}|Y_{0}^{i,N}|^2+M_k,
		\end{equation*}
		where 
		\begin{equation*}
			\begin{split}
				M_k&=\frac{1}{N}\sum_{l=1}^{k-1}\frac{B_{l}}{\big(1+{\vartheta_{\Delta}^{*}}^\Delta(b_2\Delta^2+c_2\Delta)\big)^{l}}\\&=\frac{1}{N} \sum_{l=1}^{k-1}\sum_{i=1}^{N}\frac{{\vartheta_{\Delta}^{*}}^{(l+1)\Delta}\gamma_{l}^{i}}{\big(1+{\vartheta_{\Delta}^{*}}^\Delta(b_2\Delta^2+c_2\Delta)\big)^{l}}.
			\end{split}
		\end{equation*}
		The following will demonstrate that $M_k$ is a martingale.

		For any $j\in\mathbb{S}_{N}$, it is worth noting that $\mathbb{E}\big[(|\Delta W_{t_{k-1}}^{j}|^2-m\Delta)|\mathcal{F}_{(k-1)\Delta}\big]=0$, so
		\begin{equation*}
			\begin{split}
				\mathbb{E}\big[\gamma_{k-1}^{j}|\mathcal{F}_{(k-1)\Delta}\big]=&	\|\sigma(Y_{t_{k-1}}^{j,N},\mu_{t_{k-1}}^{Y,N})\|^2\mathbb{E}\big[(|\Delta W_{t_{k-1}}^{j}|^2-m\Delta)|\mathcal{F}_{(k-1)\Delta}\big]\\&+2\mathbb{E}\big[\langle Y_{t_{k-1}}^{j,N},\sigma(Y_{t_{k-1}}^{j,N},\mu_{t_{k-1}}^{Y,N})\Delta W_{t_{k-1}}^{j}\rangle|\mathcal{F}_{(k-1)\Delta}\big]\\&
				+2\mathbb{E}\big[\langle b(Y_{t_{k-1}}^{j,N},\mu_{t_{k-1}}^{Y,N})\Delta,\sigma(Y_{t_{k-1}}^{j,N},\mu_{t_{k-1}}^{Y,N})\Delta W_{t_{k-1}}^{j}\rangle|\mathcal{F}_{(k-1)\Delta}\big]\\
				=&0.
			\end{split}
		\end{equation*}
		Then,
		\begin{equation*}
			\begin{split}
				\mathbb{E}\big[M_k|\mathcal{F}_{(k-1)\Delta}\big]=&	\mathbb{E}\big[M_{k-1}+\frac{1}{N}\sum_{i=1}^{N}\frac{{\vartheta_{\Delta}^{*}}^{k\Delta}\gamma_{k-1}^{i}}{{\big(1+{\vartheta_{\Delta}^{*}}^\Delta(b_2\Delta^2+c_2\Delta)\big)^{k-1}}}\big|\mathcal{F}_{(k-1)\Delta}\big]\\
				=&M_{k-1}+\frac{1}{N}\sum_{i=1}^{N}\frac{{\vartheta_{\Delta}^{*}}^{k\Delta}}{{\big(1+{\vartheta_{\Delta}^{*}}^\Delta(b_2\Delta^2+c_2\Delta)\big)^{k-1}}}\mathbb{E}\big[\gamma_{k-1}^{i}|\mathcal{F}_{(k-1)\Delta}\big]\\
				=&M_{k-1}.
			\end{split}
		\end{equation*}
		Therefore, applying Lemma 2.2 in \cite{mao1997} with (\ref{sure}) gives that for any $j\in\mathbb{S}_{N}$ and $\Delta<\bar{\Delta}_{1}\wedge\bar{\Delta}_{2}\wedge 1$,
		\begin{equation}\label{as}
			\lim_{k\rightarrow\infty}\iota^{*}_{k}\frac{1}{N}\sum_{i=1}^{N}|Y_{t_{k}}^{i,N}|^2\leq\lim_{k\rightarrow\infty}H_{k}(\omega)<\infty~~~~a.s.
		\end{equation}
		%where $\eta^{*}_{k}=\frac{{\vartheta_{\Delta}^{*}}^{k\Delta}}{e^{(k-1)\Delta {\vartheta_{\Delta}^{*}}^\Delta(b_2\Delta+a_2)}}$ and ${\vartheta_{\Delta}^{*}}$ is a positive constant which is larger than $e$. 
		where $\iota^{*}_{k}={\vartheta_{\Delta}^{*}}^{k\Delta}/ \exp\big({k\Delta {\vartheta_{\Delta}^{*}}^\Delta(b_2\Delta+c_2)}\big)$. 
		Choose $\tau_{\Delta}^{*}$ and $\xi_{\Delta}^{*}$ such that $\vartheta_{\Delta}^{*}=\exp(\tau_{\Delta}^{*})$ and $\iota^{*}_{k}=\exp(k\Delta\xi_{\Delta}^{*})$, then (\ref{as}) reads
		%Let ${\tau_{\Delta}^{*}}=\log{\vartheta_{\Delta}^{*}}$ and ${\xi_{\Delta}^{*}}=\frac{\log{\eta^{*}_{k}}}{k\Delta}$, then (\ref{as}) reads
		\begin{equation*}
			\lim_{k\rightarrow\infty}\exp(k\Delta\xi_{\Delta}^{*})\frac{1}{N}\sum_{i=1}^{N}|Y_{t_{k}}^{i,N}|^2<\infty.
		\end{equation*}
		Define
		\begin{equation*}
			\tilde{f}(\lambda)=b_1\Delta-c_1+\frac{(1-\lambda^{-\Delta})}{\Delta}.
		\end{equation*}
		%Even from $(\ref{f})$, ${\vartheta_{\Delta}^{*}}$ can be represented by $e^{\frac{-\log(1+b_1\Delta^2-a_1\Delta)}{\Delta}}$.
		%*************************
		%
		%\begin{rem}
		%In order to derive the rate of stability, an additional constraint exerted on stepsize is $\Delta<\frac{a_1}{b_1}$.
		%\end{rem}
		%$K_{\Delta}^{*}=e+\epsilon_1$, $\epsilon_1$ is a positive constant.
		%%%%%%%%%%%%%%%
		%Moreover, $\eta^{*}_{k}$ can be expressed as: %$$\frac{(\frac{1}{1+b_1\Delta^2-a_1\Delta})^k}{e^{(k-1)\frac{b_2\Delta^2+a_2\Delta}{1+b_1\Delta^2-a_1\Delta}}}.$$
		%$$(\frac{1}{1+b_1\Delta^2-a_1\Delta})^k\big/e^{(k-1)\frac{b_2\Delta^2+a_2\Delta}{1+b_1\Delta^2-a_1\Delta}}.$$
		%Denote $w^k_{\Delta}:=(\frac{1}{1+b_1\Delta^2-a_1\Delta})^k$ and $v^k_{\Delta}:=e^{(k-1)\frac{b_2\Delta^2+a_2\Delta}{1+b_1\Delta^2-a_1\Delta}}$, then
		%\begin{equation}\label{e}
		%	\begin{split}
			%\lim_{\Delta\rightarrow0}\lim_{k\rightarrow\infty}w^k_{\Delta}=&\lim_{\Delta\rightarrow0}\lim_{k\rightarrow\infty}\big(1+\frac{a_1\Delta-b_1\Delta^2}{1+b_1\Delta^2-a_1\Delta}\big)^k\\
			%=&\lim_{\Delta\rightarrow0}\lim_{k\rightarrow\infty}e^{\frac{k(a_1\Delta-b_1\Delta^2)}{1+b_1\Delta^2-a_1\Delta}}.
			%	\end{split}
		%\end{equation}
		According to $f(\vartheta_{\Delta}^{*})=0$, we obviously have 
		\begin{equation*}
			\tilde{f}(\vartheta_{\Delta}^{*})=b_1\Delta-c_1+\frac{(1-{\vartheta_{\Delta}^{*}}^{-\Delta})}{\Delta}=	b_1\Delta-c_1+\frac{(1-\exp(-\Delta{\tau_{\Delta}^{*}}))}{\Delta}=0.
		\end{equation*}
		Thus, 
		\begin{equation*}
			\lim_{\Delta\rightarrow0}\tau_{\Delta}^{*}	=c_1.
		\end{equation*}
		From (\ref{f}) and the chosen $\tau_{\Delta}^{*}$ and ${\xi_{\Delta}^{*}}$, we therefore have that
		\begin{equation*}
			\begin{split}
				{\xi_{\Delta}^{*}}={\tau_{\Delta}^{*}}- {\vartheta_{\Delta}^{*}}^\Delta(b_2\Delta+c_2)=-\frac{\log(1+b_1\Delta^2-c_1\Delta)}{\Delta}-\frac{b_2\Delta+c_2}{1+b_1\Delta^2-c_1\Delta}.
			\end{split}
		\end{equation*}
		Define
		\begin{equation*}
			g(\Delta)=\log(1+b_1\Delta^2-c_1\Delta)+\frac{b_2\Delta^2+c_2\Delta}{1+b_1\Delta^2-c_1\Delta}.
		\end{equation*}
		Then, 
		\begin{equation*}
			\begin{split}
				g'(\Delta)
				=&\frac{2{b_1}^2\Delta^3-(b_1c_2+c_1b_2+3b_1c_1)\Delta^2+(2b_1+2b_2+{c_1}^2)\Delta+c_2-c_1}{(1+b_1\Delta^2-c_1\Delta)^2}.
			\end{split}
		\end{equation*}
		%=&\frac{2b_1\Delta-c_1}{1+b_1\Delta^2-c_1\Delta}+\frac{(2b_2\Delta+c_2)(1+b_1\Delta^2-c_1\Delta)-(2b_1\Delta-c_1)(b_2\Delta^2+c_2\Delta)}{(1+b_1\Delta^2-c_1\Delta)^2}\\
		Thus, there exists a $\bar{\Delta}_3$ such that $g'(\Delta)<0$ for $\Delta\in(0,\bar{\Delta}_3)$. Let $\bar{\Delta}_0=\bar{\Delta}_{1}\wedge\bar{\Delta}_{2}\wedge\bar{\Delta}_{3}\wedge 1$. Since $g(0)=0$ and $g'(\Delta)<0$ for $\Delta\in(0,\bar{\Delta}_0)$, we have $g(\Delta)<0$ for $\Delta\in(0,\bar{\Delta}_0)$, which implies $\xi_{\Delta}^{*}>0$.\\
		It follows that
		\begin{equation*}
			\lim_{\Delta\rightarrow0}	{\xi_{\Delta}^{*}}=\lim_{\Delta\rightarrow0}{\tau_{\Delta}^{*}}-\lim_{\Delta\rightarrow0}\big( {\vartheta_{\Delta}^{*}}^\Delta(b_2\Delta+c_2)\big)=c_1-c_2.
		\end{equation*}
		Then, we derive the desired statement.
		%
		%
		%
		%\begin{equation}
		%	\begin{split}
			%		\lim_{\Delta\rightarrow0}\lim_{k\rightarrow\infty}e^{\xi_{\Delta} k\Delta}\frac{1}{N}\sum_{i=1}^{N}|Y_{t_{k}}^{i,N}|^2<\infty~~~~a.s.
			%	\end{split}
		%\end{equation}
		%where $\xi_{\Delta}=(a_1-a_2)k\Delta-(b_1+b_2)k\Delta^2+(b_2\Delta^2+a_2\Delta)$.

		%\begin{equation*}
		%	\begin{split}
			%		\lim_{\Delta\rightarrow0}\lim_{k\rightarrow\infty}\frac{1}{k\Delta}\log\big(\frac{1}{N}\sum_{i=1}^{N}|Y_{t_{k}}^{i,N}|^2\big)<a_2-a_1~~~~a.s.
			%	\end{split}
		%\end{equation*}
	\end{proof}

	\section{Superlinear coefficient: stability of  backward EM scheme}
	%	Furthermore, using the proof techniques of Theorems $3.1$, $5.2$ in \cite{qiaosiam} with the theories about pathwise and law uniqueness in \cite{dingqiao,1991}, we can obtain the following theorem. 
	For SMVEs with coefficients of superlinear growth, the EM scheme fails to preserve the stability. To address this issue, we propose the backward EM scheme for the associated particle system in this section, and prove that it maintains mean-square and almost sure stabilities under such conditions.
	%In this section, we construct the backward EM scheme for the associated particle system and investigate its stability properties.
	Define the time-step $\Delta=\frac{1}{M}$ for a positive integer $M$.
	For each $j\in\mathbb{S}_{N}$, the backward EM scheme for interacting particle system  (\ref{MVVVVV}) is defined by:
	\begin{equation}\label{bcnum}
		Z_{t_{k+1}}^{j,N}=Z_{t_{k}}^{j,N}+b(Z_{t_{k+1}}^{j,N},\mu_{t_{k}}^{Z,N})\Delta+\sigma(Z_{t_{k}}^{j,N},\mu_{t_{k}}^{Z,N})\Delta W_{t_k}^{j},
	\end{equation}
	where $Z_{0}^{j,N}=X_{0}^{j}$, $\Delta W_{t_k}^{j}=W_{t_{k+1}}^{j}-W_{t_k}^{j}$ and $\mu_{t_{k}}^{Z,N}(\cdot):=\frac{1}{N}\sum_{i=1}^{N}\delta_{Z_{t_{k}}^{i,N}}(\cdot)$.
	\begin{ass}\label{ass3}
		For any $\mu\in\mathcal{P}_2(\mathbb{R}^d)$, there exists a positive constant $C$ such that
		$
		|b(0,\mu)|+\|\sigma(0,\mu)\|\leq C.
		$
		%		\begin{equation*}
			%			|b(0,\delta_{0})|+\|\sigma(0,\delta_{0})\|=0,
			%		\end{equation*}
	\end{ass}
	It is easy to see that the backward EM scheme (\ref{bcnum}) is well defined under Assumptions \ref{ass2}, \ref{ass1}, and \ref{ass3}, see  \cite{31}.
	%	\section{Mean-square stability of backward EM scheme for interacting particle system}
	\subsection{Mean-square stability of backward EM scheme}
	In this subsection, the mean-square stability of the backward EM scheme is obtained even if the state variable in diffusion coefficient is superlinear.
	\begin{ass}\label{asas0}
		Assume that $\sigma(x,\mu)=\sigma_1(x)+\sigma_2(\mu)$.
		For any $x\in\mathbb{R}^d$ and $\mu\in\mathcal{P}_2(\mathbb{R}^d)$,
		there exist constants $l_1,l_2,d_2>0$ and $p_0\geq3$ such that
		\begin{equation*}
			2\langle x,b(x,\mu)\rangle+(p_0-1)\|\sigma_1(x)\|^2\leq -l_1|x|^2+l_2\mathcal{W}_{2}(\mu)^2,
		\end{equation*}
		\begin{equation*}
			\|\sigma_2(\mu)\|^2\leq d_2\mathcal{W}_{2}(\mu)^2.
		\end{equation*}
	\end{ass}
	
	Obviously, under Assumptions \ref{ass2}, \ref{ass3}, and \ref{asas0}, the state variables in both drift and diffusion coefficients are allowed to be highly nonlinear. Then, we can get the following theorem.
	%It should be noticed that taking appropriate values in Assumption \ref{asas} can make Assumption \ref{ass1} hold.

	\begin{thm}\label{mean}
		Let Assumptions \ref{ass2}, \ref{ass3}, and \ref{asas0} hold with $l_1>l_2+2d_2$. Then, for any $j\in\mathbb{S}_{N}$ and $\Delta\in(0,1)$,	the numerical solution $Z_{t_{k}}^{j,N}$ is  exponentially stable in mean-square sense, i.e.,
		$$\lim_{k\rightarrow\infty}\frac{1}{k\Delta}\log\big(	\mathbb{E}|Z_{t_{k}}^{j,N}|^2\big)\leq -\alpha_{\Delta}^{*},$$
		where $\alpha_{\Delta}^{*}$ is a positive constant satisfying
		$\lim_{\Delta\rightarrow0}{\alpha_{\Delta}^{*}}=l_1-l_2-2d_2.$
		% Here, the constraint on $\Delta_0$ is given in the proof.
	\end{thm}
	\begin{proof}
		By implementing an elementary equality on (\ref{bcnum}), we obtain that
		\begin{equation}\label{009}
			\begin{split}
				&2\langle Z_{t_{k+1}}^{j,N},Z_{t_{k+1}}^{j,N}-Z_{t_{k}}^{j,N}\rangle\\
				=&|Z_{t_{k+1}}^{j,N}|^2-|Z_{t_{k}}^{j,N}|^2+|Z_{t_{k+1}}^{j,N}-Z_{t_{k}}^{j,N}|^2\\
				%	=&2\Delta\langle Z_{t_{k+1}}^{j,N},b(Z_{t_{k+1}}^{j,N},\mu_{t_{k}}^{Z,N})\rangle+2\langle Z_{t_{k+1}}^{j,N},  \sigma(Z_{t_{k}}^{j,N},\mu_{t_{k}}^{Z,N})\Delta W_{t_k}^{j}\rangle\\
				=&2\Delta\langle Z_{t_{k+1}}^{j,N},b(Z_{t_{k+1}}^{j,N},\mu_{t_{k}}^{Z,N})\rangle+2\langle Z_{t_{k+1}}^{j,N}-Z_{t_{k}}^{j,N}, \sigma(Z_{t_{k}}^{j,N},\mu_{t_{k}}^{Z,N})\Delta W_{t_k}^{j}\rangle,
			\end{split}
		\end{equation}
		which means 
		$$
		\mathbb{E}|Z_{t_{k+1}}^{j,N}|^2-{E}|Z_{t_{k}}^{j,N}|^2\leq2\Delta{E}\langle Z_{t_{k+1}}^{j,N},b(Z_{t_{k+1}}^{j,N},\mu_{t_{k}}^{Z,N})\rangle+\Delta{E}\| \sigma(Z_{t_{k}}^{j,N},\mu_{t_{k}}^{Z,N})\|^2,
		$$
		with Young's inequality.	
		From Assumption \ref{asas0}, one can see that
		\begin{equation}\label{202}
			\begin{split}
				\mathbb{E}|Z_{t_{k+1}}^{j,N}|^2-\mathbb{E}|Z_{t_{k}}^{j,N}|^2\leq&-l_1\Delta\mathbb{E}|Z_{t_{k+1}}^{j,N}|^2+l_2\Delta\mathbb{E}\big[\frac{1}{N}\sum_{i=1}^{N}|Z_{t_{k}}^{i,N}|^2\big]-(p_0-1)\Delta\mathbb{E}\|\sigma_1(Z_{t_{k+1}}^{j,N})\|^2\\&+2\Delta\mathbb{E}\|\sigma_1(Z_{t_{k}}^{j,N})\|^2+2d_2\Delta\mathbb{E}\big[\frac{1}{N}\sum_{i=1}^{N}|Z_{t_{k}}^{i,N}|^2\big].
			\end{split}
		\end{equation}
		Summing up (\ref{202}) from 0 to $k$ yields that
		\begin{equation}\label{065}
			\begin{split}
				\mathbb{E}|Z_{t_{k}}^{j,N}|^2\leq&\mathbb{E}|Z_{{0}}^{j,N}|^2-l_1\Delta\sum_{r=0}^{k-1}\mathbb{E}|Z_{t_{r+1}}^{j,N}|^2+l_2\Delta\sum_{r=0}^{k-1}\mathbb{E}|Z_{t_{r}}^{j,N}|^2-(p_0-1)\Delta\sum_{r=0}^{k-1}\mathbb{E}\|\sigma_1(Z_{t_{r+1}}^{j,N})\|^2\\&+2\Delta\sum_{r=0}^{k-1}\mathbb{E}\|\sigma_1(Z_{t_{r}}^{j,N})\|^2+2d_2\Delta\sum_{r=0}^{k-1}\mathbb{E}|Z_{t_{r}}^{j,N}|^2.
			\end{split}
		\end{equation}
		By rearranging (\ref{065}), we obtain
		\begin{equation*}
			\begin{split}
				&	(1+l_1\Delta)	\mathbb{E}|Z_{t_{k}}^{j,N}|^2+(p_0-1)\mathbb{E}\|\sigma_1(Z_{t_{k}}^{j,N})\|^2\\\leq&(1+l_1\Delta)\mathbb{E}|Z_{{0}}^{j,N}|^2+2\Delta\mathbb{E}\|\sigma_1(Z_{{0}}^{j,N})\|^2-(p_0-3)\Delta\sum_{r=1}^{k-1}\mathbb{E}\|\sigma_1(Z_{t_{r}}^{j,N})\|^2\\&+(l_2\Delta+2d_2\Delta-l_1\Delta)\sum_{r=0}^{k-1}\mathbb{E}|Z_{t_{r}}^{j,N}|^2.
			\end{split}
		\end{equation*}
		An application of the discrete-type Gronwall inequality leads to the result that
		\begin{equation*}
			\begin{split}
			&	\mathbb{E}|Z_{t_{k}}^{j,N}|^2\\\leq&\left((1+l_1\Delta)\mathbb{E}|Z_{{0}}^{j,N}|^2+2\Delta\mathbb{E}\|\sigma_1(Z_{{0}}^{j,N})\|^2\right)\exp\left((l_2+2d_2-l_1)k\Delta\right),
			\end{split}
		\end{equation*}
		which implies that the rate of mean-square stability is $l_1-l_2-2d_2$ as $\Delta\rightarrow0$.
	\end{proof}
	\subsection{Almost sure stability of backward EM scheme }
	To derive the almost sure exponential stability of backward EM scheme, the following assumption which relates to the dimension of Brownian motion is needed.
	\begin{ass}\label{aqsqs}
		For any $x\in\mathbb{R}^d$ and $\mu\in\mathcal{P}_2(\mathbb{R}^d)$, there exist constants ${\tilde{c}_1},{\tilde{c}_2},h_1,h_2>0$ such that
		\begin{equation*}
			2\langle x,b(x,\mu)\rangle\leq-{\tilde{c}_1}|x|^2+{\tilde{c}_2}\mathcal{W}_{2}(\mu)^{2},~~~	m\|\sigma(x,\mu)\|^2\leq h_1|x|^2+h_2\mathcal{W}_{2}(\mu)^{2},
		\end{equation*}
		%		\begin{equation*}
			%			m\|\sigma(x,\mu)\|^2\leq h_1|x|^2+h_2\mathcal{W}_{2}(\mu)^{2},
			%		\end{equation*}
		where $m$ is the dimension of Brownian motion.
	\end{ass}
	
	\begin{thm}\label{almostbc}
		Let Assumptions \ref{ass2},  \ref{ass3}, and \ref{aqsqs}  hold with $ {\tilde{c}_1}>h_1+{\tilde{c}_2}+h_2$. Then, there
		exists a stepsize
		$\tilde{\Delta}_{0}\in (0,1)$ such that
		for any $\Delta\in(0,\tilde{\Delta}_{0})$, the numerical solution $Z_{t_{k}}^{j,N}$ is almost surely exponentially stable, i.e., 
		$$\lim_{k\rightarrow\infty}\frac{1}{k\Delta}\log	\big(\frac{1}{N}\sum_{i=1}^{N}|Z_{t_{k}}^{i,N}|^2\big)\leq -{\beta_{\Delta}^{*}},$$
		with ${\beta_{\Delta}^{*}}$ fulfilling
		$
		\lim_{\Delta\rightarrow0}	{\beta_{\Delta}^{*}}={\tilde{c}_1}-h_1-\tilde{c}_2-h_2.
		$
		The constraint of $\tilde{\Delta}_{1}$ is given in the proof.
	\end{thm}

	\begin{proof}
		From (\ref{bcnum}), it is easy to obtain that 
		\begin{equation*}
			\begin{split}
				|Z_{t_{k+1}}^{j,N}|^2=&\langle Z_{t_{k+1}}^{j,N},Z_{t_{k}}^{j,N}+b(Z_{t_{k+1}}^{j,N},\mu_{t_{k}}^{Y,N})\Delta+\sigma(Z_{t_{k}}^{j,N},\mu_{t_{k}}^{Y,N})\Delta W_{t_k}^{j}\rangle\\\leq&|Z_{t_{k}}^{j,N}|^2+2\Delta\langle Z_{t_{k+1}}^{j,N},b(Z_{t_{k+1}}^{j,N},\mu_{t_{k}}^{Y,N})\rangle+m\Delta\|\sigma(Z_{t_{k}}^{j,N},\mu_{t_{k}}^{Y,N})\|^2 +\chi_{k}^{j},
			\end{split}
		\end{equation*}
		where 
		$
		\chi_{k}^{j}=	\|\sigma(Z_{t_{k}}^{j,N},\mu_{t_{k}}^{Y,N})\|^2(|\Delta W_{t_k}^{j}|^2 -m\Delta)	+2\langle Z_{t_{k}}^{j,N},\sigma(Z_{t_{k}}^{j,N},\mu_{t_{k}}^{Y,N})\Delta W_{t_k}^{j}\rangle.
		$
		Using Assumption \ref{aqsqs} gives
		\begin{equation}\label{as111}
			\begin{split}
				&	|Z_{t_{k+1}}^{j,N}|^2\\\leq&|Z_{t_{k}}^{j,N}|^2-{\tilde{c}_1}\Delta|Z_{t_{k+1}}^{j,N}|^2+\frac{{\tilde{c}_2}\Delta}{N}\sum_{i=1}^{N}|Z_{t_k}^{i,N}|^2+h_1\Delta|Z_{t_k}^{j,N}|^2
				+\frac{h_2\Delta}{N}\sum_{i=1}^{N}|Z_{t_k}^{i,N}|^2+\chi_{k}^{j}.
				%	\leq&|Z_{t_{k}}^{j,N}|^2+\Delta(b_1\Delta-{\tilde{c}_1})|Z_{t_k}^{j,N}|^2+\Delta(b_2\Delta+{\tilde{c}_2})\frac{1}{N}\sum_{i=1}^{N}|Z_{t_k}^{i,N}|^2+\chi_{k}^{j}.
			\end{split}
		\end{equation}
		By (\ref{as111}), we have that for any $\lambda>1$,
		\begin{equation*}
			\begin{split}
				&(1+{\tilde{c}_1}\Delta)\lambda^{k\Delta}	|Z_{t_{k}}^{j,N}|^2	\\\leq&(1+{\tilde{c}_1}\Delta)|Z_{0}^{j,N}|^2+
				\sum_{r=0}^{k-1}(1+h_1\Delta-{\tilde{c}_1}\Delta-\lambda^{-\Delta})\lambda^{(r+1)\Delta}|Z_{t_{r}}^{j,N}|^2\\&+ \sum_{r=0}^{k-1}({\tilde{c}_2}\Delta+h_2\Delta)\lambda^{(r+1)\Delta}\frac{1}{N}\sum_{i=1}^{N}|Z_{t_r}^{i,N}|^2+\sum_{r=0}^{k-1}\lambda^{(r+1)\Delta}\chi_{r}^{j}.
			\end{split}
		\end{equation*}
		Let us introduce the function
		\begin{equation}\label{fbc}
			\tilde{g}(\lambda)=\lambda^{\Delta}(1+h_1\Delta-{\tilde{c}_1}\Delta)-1.
		\end{equation}
%		For any $\lambda>1$, one can see $	\tilde{g}'(\lambda)>0$ and $\lim_{\lambda\rightarrow{\infty}}g(\lambda)={\infty}$. And we have $\tilde{g}(1)<0$ with $h_1<{\tilde{c}_1}$.
	For any $\lambda>1$, one can see that $\tilde{g}(1)<0$ with $h_1<{\tilde{c}_1}$. Moreover, there has a $\tilde{\Delta}_0$ such that $\tilde{g}'(\lambda)>0$ and $\lim_{\lambda\rightarrow{\infty}}g(\lambda)={\infty}$ for $\Delta<\tilde{\Delta}_0\wedge1$.
		Thus, there exists a unique ${{\eta_{\Delta}^{*}}}>1$ which is related to the stepsize $\Delta$, such that $	\tilde{g}({{\eta_{\Delta}^{*}}})=0$. By taking $\lambda={{\eta_{\Delta}^{*}}}$ and summing $j$ from $1$ to $N$, we get the result
		\begin{equation*}
			\begin{split}
				&	\sum_{i=1}^{N}	{{\eta_{\Delta}^{*}}}^{k\Delta}	|Z_{t_{k}}^{i,N}|^2
				\\	\leq& (1+{\tilde{c}_1}\Delta)\sum_{i=1}^{N}|Z_{0}^{i,N}|^2+({\tilde{c}_2}\Delta+h_2\Delta)\sum_{r=0}^{k-1}\sum_{i=1}^{N}	{{\eta_{\Delta}^{*}}}^{(r+1)\Delta}|Z_{t_{r}}^{i,N}|^2+\sum_{r=0}^{k-1}\sum_{i=1}^{N}{{\eta_{\Delta}^{*}}}^{(r+1)\Delta}\chi_{r}^{i}.
			\end{split}
		\end{equation*}
		For $r\in\{0,1,\cdots,k-1\}$, define $U_r=	\sum_{i=1}^{N}{{\eta_{\Delta}^{*}}}^{r\Delta}	|Z_{t_{r}}^{i,N}|^2$, $P_r=\sum_{i=1}^{N}{{\eta_{\Delta}^{*}}}^{(r+1)\Delta}\chi_{r}^{i}$, and
		\begin{equation*}
			\begin{split}
				\varLambda_k
				= (1+{\tilde{c}_1}\Delta)\sum_{i=1}^{N}|Z_{0}^{i,N}|^2+{\eta_{\Delta}^{*}}^\Delta({\tilde{c}_2}\Delta+h_2\Delta)\sum_{r=0}^{k-1}U_r+\sum_{r=0}^{k-1}P_r,
			\end{split}
		\end{equation*}
		which means $U_r\leq \varLambda_r$ for $r\in\{0,1,\cdots,k\}$. Then,
		\begin{equation}\label{interbc}
			\begin{split}
				\varLambda_k=&\varLambda_{k-1}
				+ {\eta_{\Delta}^{*}}^\Delta({\tilde{c}_2}\Delta+h_2\Delta)U_{k-1}+P_{k-1}
				\\	\leq& \big(1+{\eta_{\Delta}^{*}}^\Delta({\tilde{c}_2}\Delta+h_2\Delta)\big)\varLambda_{k-1}+P_{k-1}.
			\end{split}
		\end{equation}
		%provided by
		%	\begin{equation}\label{Xx}
			%		\begin{split}
				%			\varLambda_k-\varLambda_{k-1}
				%			= {{\eta_{\Delta}^{*}}}^\Delta(e_2\Delta+h_2\Delta)U_{k-1}+P_{k-1}.
				%		\end{split}
			%	\end{equation}
		Iterating (\ref{interbc}) gives that
		$$
		\varLambda_k
		\leq \big(1+{\eta_{\Delta}^{*}}^\Delta({\tilde{c}_2}\Delta+h_2\Delta)\big)^{k-1}\varLambda_{1}+\sum_{r=0}^{k-1}\big(1+{\eta_{\Delta}^{*}}^\Delta({\tilde{c}_2}\Delta+h_2\Delta)\big)^{k-1-r}P_{r}.
		$$
		Thus,
		\begin{equation}\label{eeebc}
			\begin{split}
				&	\sum_{i=1}^{N}{\eta_{\Delta}^{*}}^{k\Delta}	|Z_{t_{k}}^{i,N}|^2\\
				\leq& \exp\big({k\Delta {\eta_{\Delta}^{*}}^\Delta({\tilde{c}_2}+h_2)}\big)	\left((1+{\tilde{c}_1}\Delta)\sum_{i=1}^{N}|Z_{0}^{i,N}|^2+\sum_{r=1}^{k-1}\big(1+{\eta_{\Delta}^{*}}^\Delta({\tilde{c}_2}\Delta+h_2\Delta)\big)^{-r}P_{r}\right).
			\end{split}
		\end{equation}
		Dividing both sides of (\ref{eeebc}) by $N \exp\big({k\Delta {\eta_{\Delta}^{*}}^\Delta({\tilde{c}_2}+h_2)}\big)$, we have
		\begin{equation}\label{surebc}
			\begin{split}
				&	\Big(\frac{{\eta_{\Delta}^{*}}^{k\Delta}}{ \exp\big({k\Delta {\eta_{\Delta}^{*}}^\Delta({\tilde{c}_2}+h_2)}\big)}\Big)\frac{1}{N}\sum_{i=1}^{N}	|Z_{t_{k}}^{i,N}|^2
				\\	\leq&(1+{\tilde{c}_1}\Delta)	\frac{1}{N}\sum_{i=1}^{N} |Z_{0}^{i,N}|^2+\frac{1}{N}\sum_{r=1}^{k-1}\big(1+{\eta_{\Delta}^{*}}^\Delta({\tilde{c}_2}\Delta+h_2\Delta)\big)^{-r}P_{r}.
			\end{split}
		\end{equation}
		Define 
		$$
		\tilde{H}_k=	(1+{\tilde{c}_1}\Delta)\frac{1}{N}\sum_{i=1}^{N}|Z_{0}^{i,N}|^2+\frac{1}{N}\sum_{r=1}^{k-1}\big(1+{\eta_{\Delta}^{*}}^\Delta({\tilde{c}_2}\Delta+h_2\Delta)\big)^{-r}P_{r}.
		$$
		%	where $U_k=\frac{1}{N} \sum_{l=1}^{k-1}\sum_{i=1}^{N}\big(1+{\eta_{\Delta}^{*}}^\Delta(e_2\Delta+h_2\Delta)\big)^{-l}{\eta_{\Delta}^{*}}^{(l+1)\Delta}\chi_{l}^{i}$.
		It is straightforward to verify that the second term in $	\tilde{H}_k$ is a martingale through conventional approaches. 
		Then, the application of semimartingale convergence theorem yields that, for any $j\in\mathbb{S}_{N}$,
		\begin{equation}\label{asbc}
			\lim_{k\rightarrow\infty}\zeta^{*}_{k}\frac{1}{N}\sum_{i=1}^{N}|Z_{t_{k}}^{i,N}|^2\leq\lim_{k\rightarrow\infty}	\tilde{H}_k(\omega)<\infty~~~~a.s.
		\end{equation}
		%where $\eta^{*}_{k}=\frac{{\eta_{\Delta}^{*}}^{k\Delta}}{e^{(k-1)\Delta {\eta_{\Delta}^{*}}^\Delta(b_2\Delta+a_2)}}$ and ${\eta_{\Delta}^{*}}$ is a positive constant which is larger than $e$. 
		with $\zeta^{*}_{k}={\eta_{\Delta}^{*}}^{k\Delta}/ \exp\big({k\Delta {\eta_{\Delta}^{*}}^\Delta({\tilde{c}_2}+h_2)}\big)$. 
		And there exist $\varrho_{\Delta}^{*}$ and $\kappa_{\Delta}^{*}$ such that $\zeta_{k}^{*}=\exp(k\Delta\varrho_{\Delta}^{*})$ and $\eta^{*}_{\Delta}=e^{\kappa_{\Delta}^{*}}$, then (\ref{asbc}) can be regarded as
		%Let ${\tau_{\Delta}^{*}}=\log{\eta_{\Delta}^{*}}$ and ${\xi_{\Delta}^{*}}=\frac{\log{\eta^{*}_{k}}}{k\Delta}$, then (\ref{as}) reads
		$$
		\lim_{k\rightarrow\infty}\exp(k\Delta\varrho_{\Delta}^{*})\frac{1}{N}\sum_{i=1}^{N}|Z_{t_{k}}^{i,N}|^2<\infty,
		$$
		with $
		{\varrho_{\Delta}^{*}}={\kappa_{\Delta}^{*}}- {\eta_{\Delta}^{*}}^\Delta({\tilde{c}_2}+h_2)$.
		By $\tilde{g}(\eta_{\Delta}^{*})=0$, we obtain
		\begin{equation*}
			h_1-{\tilde{c}_1}+\frac{(1-{\eta_{\Delta}^{*}}^{-\Delta})}{\Delta}=		h_1-{\tilde{c}_1}+\frac{(1-\exp(-\Delta{\kappa_{\Delta}^{*}}))}{\Delta}=0,
		\end{equation*}
		which implies
		$
		\lim_{\Delta\rightarrow0}\kappa_{\Delta}^{*}	={\tilde{c}_1}-h_1.
		$
	By (\ref{fbc}), ${\varrho_{\Delta}^{*}}$ can be written as:
			$$
	\varrho_{\Delta}^{*}=\kappa_{\Delta}^{*}-\frac{{\tilde{c}_2}+h_2}{1+h_1\Delta-{\tilde{c}_1}\Delta}.
	$$
	Then, 
	$$
\lim_{\Delta\rightarrow0}	\varrho_{\Delta}^{*}=\lim_{\Delta\rightarrow0}\kappa_{\Delta}^{*}-\lim_{\Delta\rightarrow0}\frac{{\tilde{c}_2}+h_2}{1+h_1\Delta-{\tilde{c}_1}\Delta}={\tilde{c}_1}-h_1-{\tilde{c}_2}-h_2.
	$$
%		Define
%		$$
%		w(\Delta)=	\log(\frac{1+{\tilde{c}_1}\Delta}{1+h_1\Delta})-\frac{({\tilde{c}_2}+h_2)(\Delta+{\tilde{c}_1}\Delta^2)}{1+h_1\Delta}.
%		$$
%		It is easy to obtain $w(0)=0$ and 
%		\begin{equation*}
%			\begin{split}
%				w'(\Delta)
%				=&\frac{-({\tilde{c}_2}+h_2)\left({{\tilde{c}_1}^2}{h_1}\Delta^3-(2{{\tilde{c}_1}^2}+{\tilde{c}_1}{h_1})\Delta^2\right)+R_{c,h}\Delta+\tilde{R}_{c,h}}{(1+{\tilde{c}_1}\Delta)(1+h_1\Delta)^2},
%			\end{split}
%		\end{equation*}
%		with $R_{c,h}:=h_1({\tilde{c}_1}-h_1)-3{\tilde{c}_1} ({\tilde{c}_2}+ h_2)$ and $\tilde{R}_{c,h}:={\tilde{c}_1}-h_1-{\tilde{c}_2}-h_2$.
%		%=&\frac{2b_1\Delta-{\tilde{c}_1}}{1+b_1\Delta^2-{\tilde{c}_1}\Delta}+\frac{(2b_2\Delta+{\tilde{c}_2})(1+b_1\Delta^2-{\tilde{c}_1}\Delta)-(2b_1\Delta-{\tilde{c}_1})(b_2\Delta^2+{\tilde{c}_2}\Delta)}{(1+b_1\Delta^2-{\tilde{c}_1}\Delta)^2}\\
%		Through a series of straightforward analysis, one can obtain that there exists a $\tilde{\Delta}_1$ such that $w'(\Delta)>0$ and $w(\Delta)>0$ for $\Delta\in(0,\tilde{\Delta}_1)$, that is $\xi_{\Delta}^{*}>0$. Moreover, we see
%		$$
%		\lim_{\Delta\rightarrow0}	{\varrho_{\Delta}^{*}}=\lim_{\Delta\rightarrow0}{\kappa_{\Delta}^{*}}-\lim_{\Delta\rightarrow0}\big( {\eta_{\Delta}^{*}}^\Delta({\tilde{c}_2}+h_2)\big)={\tilde{c}_1}-h_1-{\tilde{c}_2}-h_2.
%		$$
		Then, the desired result can be obtained.
	\end{proof}

	\section{Numerical examples}
	
	In this section, four numerical examples are demonstrated to highlight the importance of the stability of the numerical scheme.
	
	%we demonstrate the application of stability of EM scheme for SMVE in the stochastic opinion dynamics model and  feedback control problem. 
	
	\begin{exa}
		{\rm
			Consider stochastic opinion dynamics model with a stubborn opinion in \cite{opinion}:
			\begin{equation}\label{exm1}
				dX_t=\big( f(t)[\bar{X}_t\textbf{1}-X_t]-g(t)X_t\big) dt+\sigma X_tdW_t,
			\end{equation}
			with random initial data $X_0\sim\mathcal{N}(2,1)$, where $\mathcal{N}(\cdot,\cdot)$ is normal distribution.
			%$\omega$ is a random variable following the normal distribution.
			
			Let $X_t=\big[X_t^1,X_t^2,\cdots,X_t^N\big]^{T}$ be a vector in which each component $X_t^i$ represents the opinion held by each individual at time $t$. And $\bar{X}_t=(1/N)\sum_{i=1}^{N}X_t^i$ denotes the average opinion. Here, $ \textbf{1}=[1,1,\cdots,1]^{T}$.
			
			The part $f(t)[\bar{X}_t\textbf{1}-X_t]$ represents the interaction between an individual and its neighbors, and the other part $-g(t)X_t$ means the interaction between an individual and the stubborn ones. Compared with additive noise, the diffusion term of (\ref{exm1}) can take fading, frequency selectivity, interference into account.
			
			Let $f(t)=1, g(t)=\frac{5}{2},\sigma=1$. This means the drift coefficient $b(x,\mu)=-\frac{7}{2}x+\mu$ and the diffusion coefficient $\sigma(x,\mu)=x$. Then, we have  
			\begin{equation*}
				2\langle x,b(x,\mu)\rangle+\|\sigma(x,\mu)\|^2\leq-5x^2+ \mathcal{W}_{2}(\mu)^{2},
			\end{equation*}
			and
			\begin{equation*}
				2\langle x-y,b(x,\mu)-b(y,\nu)\rangle+\|\sigma(x,\mu)-\sigma(y,\nu)\|^2\leq-5|x-y|^2+\mathbb{W}_{2}(\mu,\nu)^{2},
			\end{equation*}
			which implies that the coefficients of (\ref{exm1}) satisfy the conditions in theorems with $b_1=7$, $b_2=2$, $a_1=5$, $a_2=1$.
			
			Therefore, on the basis of theories in the above sections, we can use EM scheme to simulate the stability of  (\ref{exm1}). The stability of (\ref{exm1}) implies that the system's response to the changes in its environment becomes less sensitive. By choosing $\Delta=10^{-2}$, which satisfies the restrictions on stepsize of the theoritical results, the simulations of interacting particle system corresponding to (\ref{exm1}) are shown below.

			To illustrate the mean-square stability result in Theorem \ref{meanE}, we simulate 1000 particles with 100 trajectories, and we observe that the sample average of particle average is stable in Figuer \ref{tu1}. 
			
			To show the almost sure stability result in Theorem \ref{almost}, we simulate 100 trajectories, and we find out that the average of 1000 particles' numerical solutions driven by each trajectory tends to zero in Figure \ref{tu2}. 
			
			% For every $j\in\mathbb{S}_N$, we compute the average of the numerical solutions simulated by 1000 sample trajectories $\mathbb{E}|X_{t}^{j,N}|^2$ and plot it in Figuer \ref{tu2}.
			
			% The mean square stability result of each particle is revealed in Figuer \ref{tu2}. 

			\begin{figure}[htbp]
				\centering	
				\begin{minipage}{0.49\linewidth}
					\centering
					\includegraphics[width=0.9\linewidth]{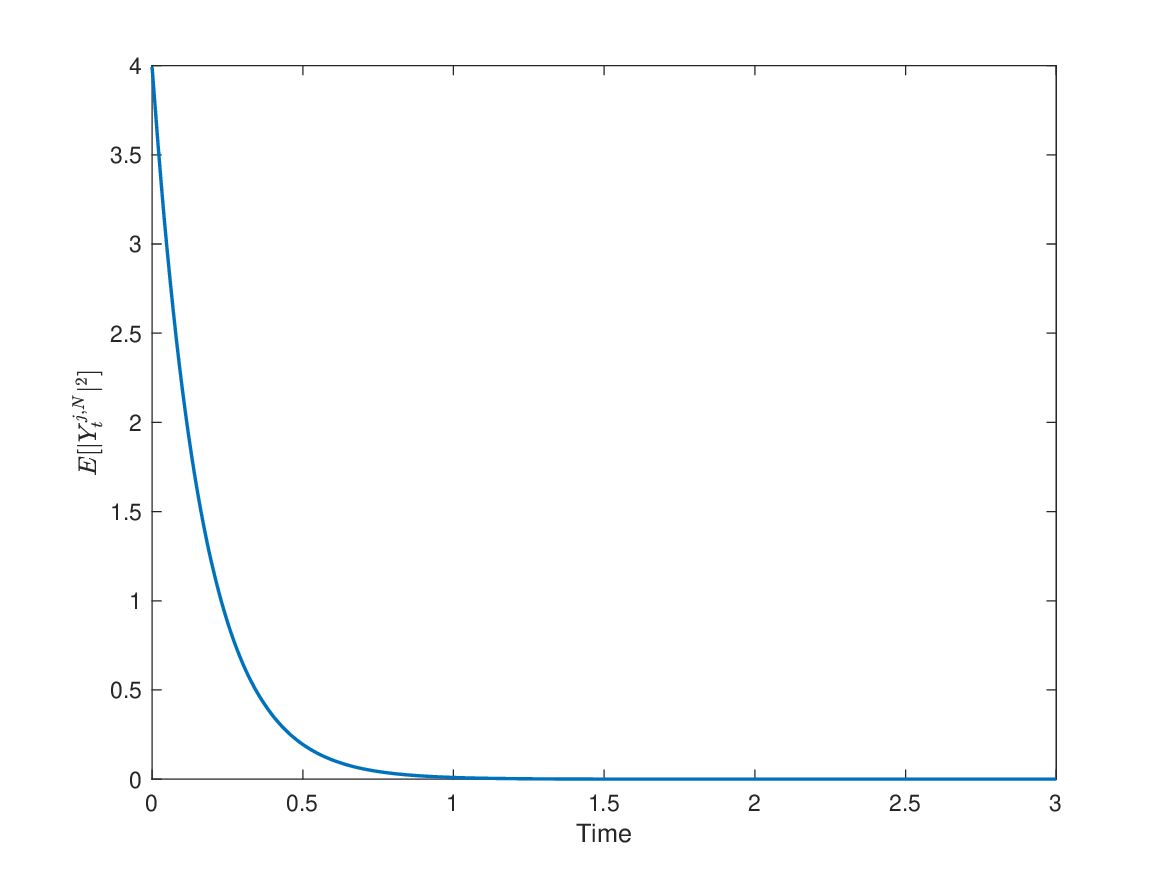}
					\caption{\label{tu1}  mean-square stability of (\ref{exm1})}
				\end{minipage}
				\begin{minipage}{0.49\linewidth}
					\centering
					\includegraphics[width=0.9\linewidth]{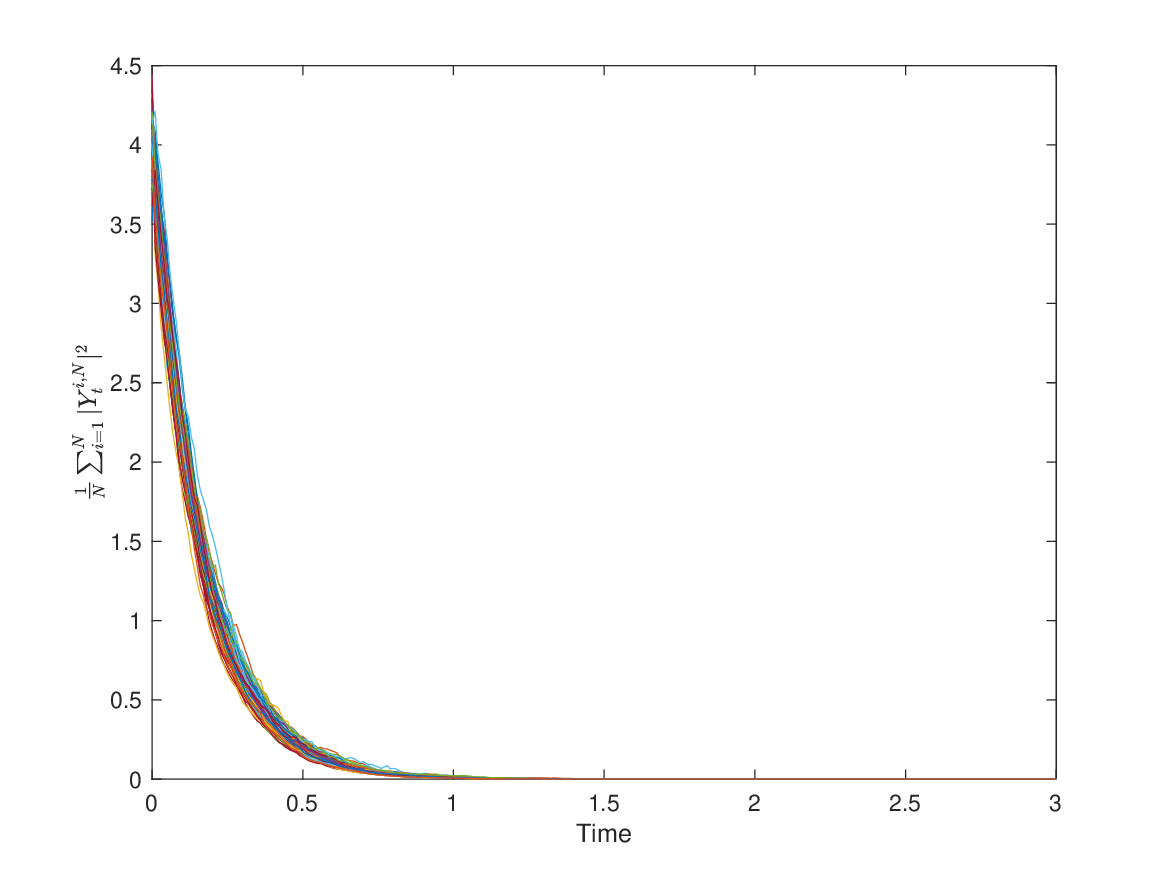}
					\caption{\label{tu2} the stability with 100 trajectories }
				\end{minipage}
				%(\ref{exm1})
			\end{figure}
			
			To check the assertion in Remark \ref{remean}, we choose three different stepsizes $\Delta=0.005,0.3,0.4$. We infer from Theorem \ref{meanE} that the rate of mean-square stability is $a_2-a_1=-4$ as $\Delta\rightarrow0$. By using 1000 particles and 100 trajectories, we find out that the rate of mean-square stability gradually approaches to $-4$ as the stepsize decreases in Figure \ref{tu0}, which supports Remark \ref{remean}.
			\begin{figure}[htbp] 
				\centering
				\includegraphics[height=7.4cm,width=9cm]{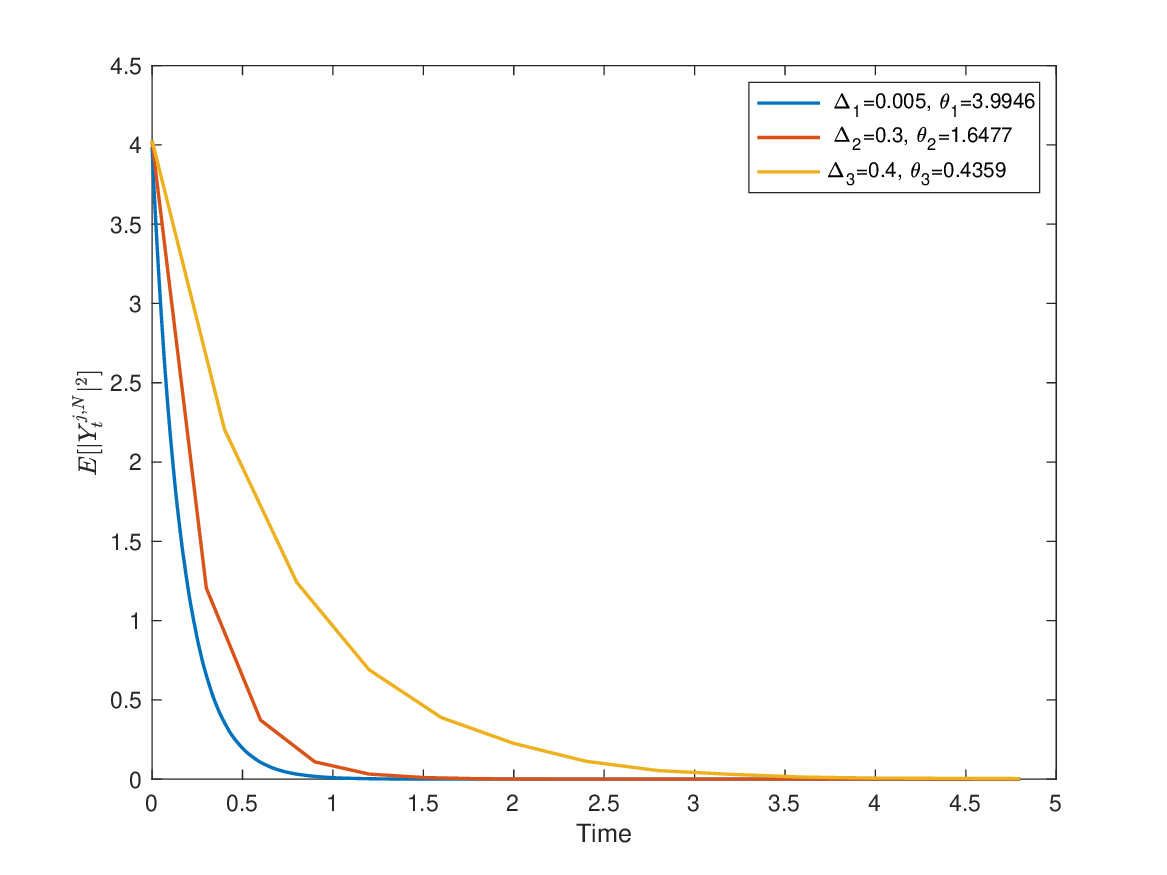}
				\caption{\label{tu0} mean-square stability with three stepsizes}
			\end{figure}
		}
	\end{exa}

	\begin{exa}
		{\rm
			Consider the following  scalar SMVE:
			\begin{equation}\label{exm3}
				dX_t=\big(-\frac{7}{2}X_t+\mathbb{E}X_t\big) dt+\big(X_t+\frac{1}{2}\mathbb{E}X_t\big)dW_t,
			\end{equation}
			with random initial data $X_0\sim\mathcal{N}(2,1)$, where $\mathcal{N}(\cdot,\cdot)$ is normal distribution.
			
			Obviously, the coefficients of (\ref{exm3}) satisfy the required assumptions.
			By means of simulating (\ref{exm3}) with stepsize $\Delta=10^{-2}$, we display the theoritical result that the rate of stability is independent of the number of particles, which is supported by Figure \ref{tu3} and Figure \ref{tu4}.
			
			From Figure \ref{tu3}, we can see that the three lines almost coincide, which exhibits that the rate of stability in mean-square sense simulated by 100 trajectories is not greatly affected by the number of particles. 
			Figure \ref{tu4} displays that the average of $N$ particles' numerical solutions driven by one trajectory tends to zero at nearly the same speed.
			
			\begin{figure}[htbp]
				\centering
				\begin{minipage}{0.49\linewidth}
					\centering
					\includegraphics[width=0.9\linewidth]{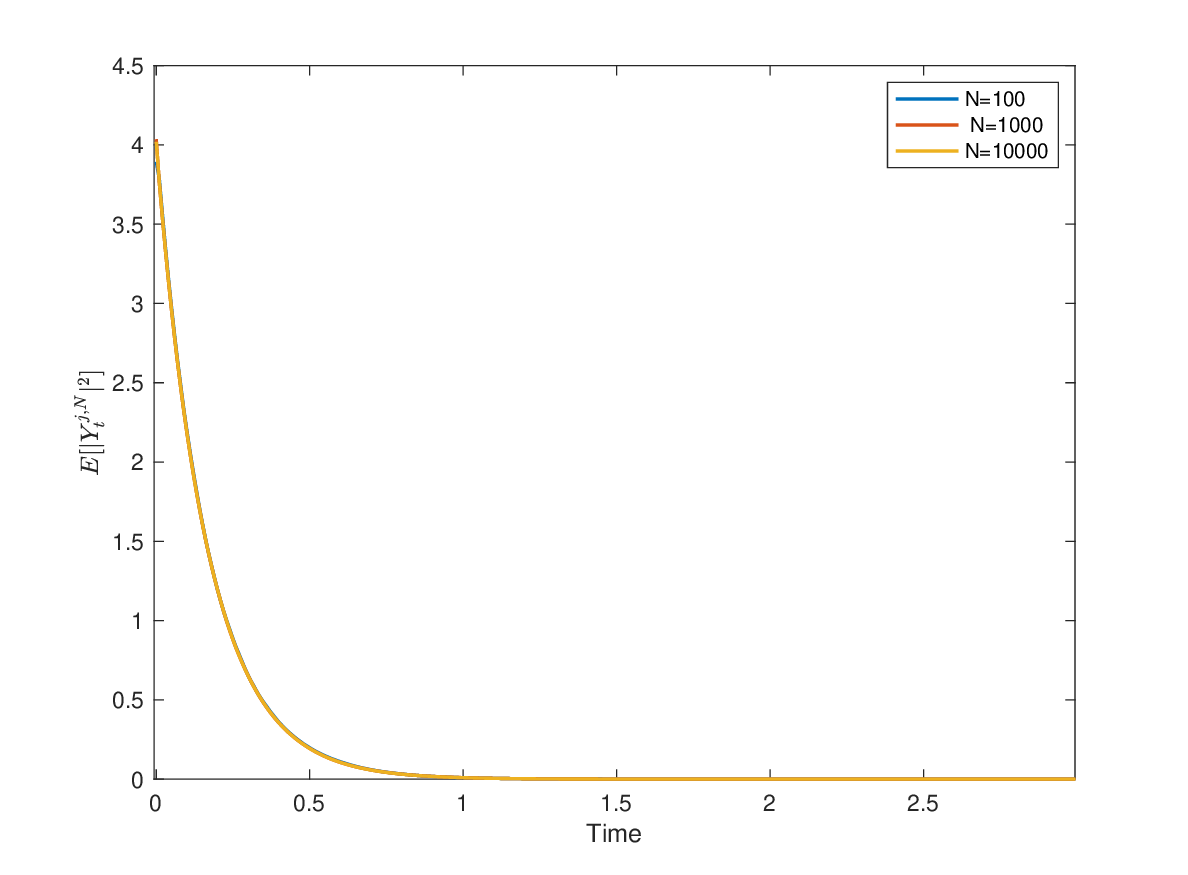}
					\caption{\label{tu3}  mean-square stability of (\ref{exm3})}
				\end{minipage}
				\begin{minipage}{0.49\linewidth}
					\centering
					\includegraphics[width=0.9\linewidth]{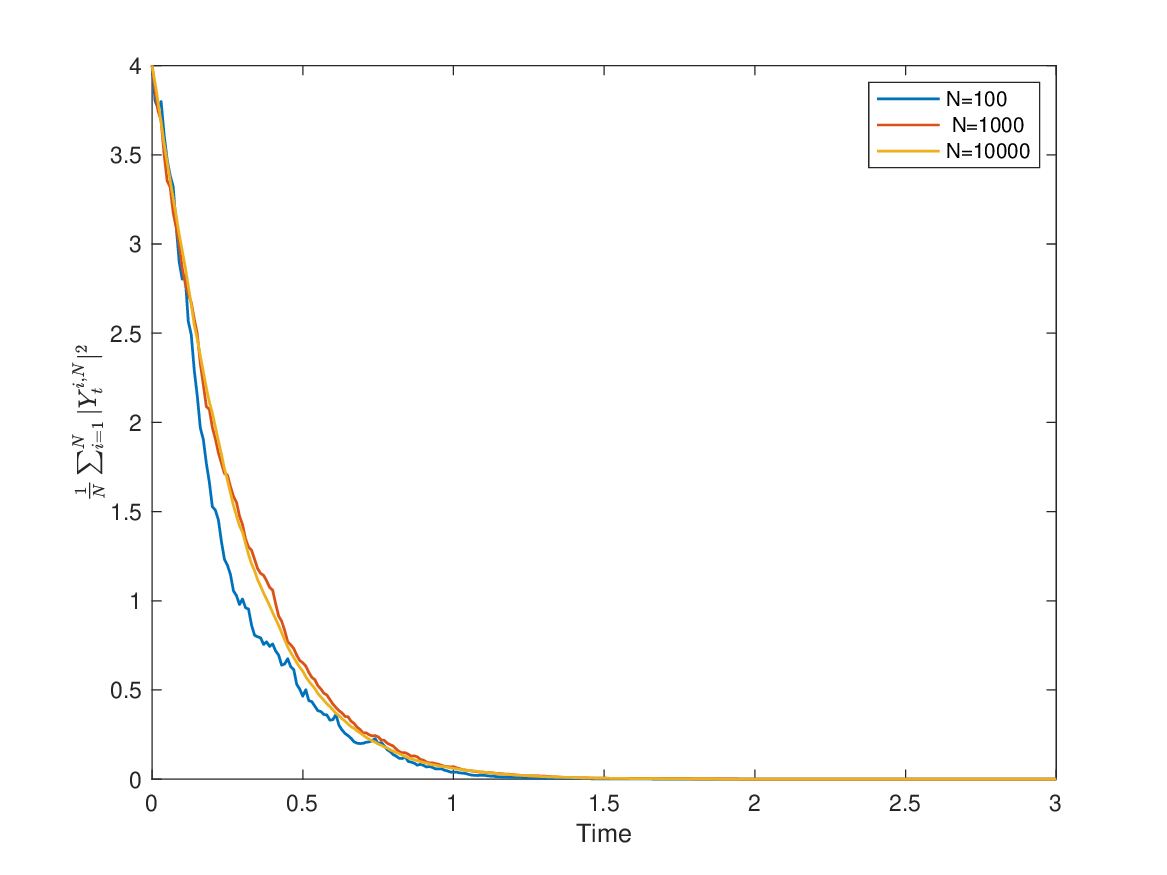}
					\caption{\label{tu4} the stability with 1 trajectory  }
				\end{minipage}
				%(\ref{exm3})
			\end{figure}
		}
	\end{exa}
	%\begin{figure}[htbp]
	%	\centering
	%	\begin{minipage}{0.49\linewidth}
		%		\centering
		%		\includegraphics[width=0.9\linewidth]{5M1000N.eps}
		%		\caption{\label{tu3} the stability of $\mathbb{E}|X_{t}^{j,N}|^2$ with different $N$}
		%	\end{minipage}
	%	%\qquad
	%	\begin{minipage}{0.49\linewidth}
		%		\centering
		%		\includegraphics[width=0.9\linewidth]{boom.eps}
		%		\caption{\label{tu4} the unstability of (\ref{exm2})}
		%	\end{minipage}
	%\end{figure}
	%
	Inspired by \cite{wuhao}, we will take advantage of the following example to illustrate that an unstable system can become stable after exerting control, which supports the theories of \cite{wuhao}. The numerical experiments in this example are simulated with 100 trajectories and 1000 particles.

	\begin{exa}
		{\rm
			Consider the following  scalar SMVE:
			\begin{equation}\label{exm2}
				dX_t=\big(2X_t+\mathbb{E}X_t\big) dt+X_tdW_t,
			\end{equation}
			with random initial data $X_0\sim\mathcal{N}(2,1)$, where $\mathcal{N}(\cdot,\cdot)$ is normal distribution.
			
			We use EM scheme with the stepsize $\Delta=10^{-2}$ to approximate the  interacting particle system w.r.t. (\ref{exm2}).
			From Figure \ref{tu5} and Figure \ref{tu6}, we can see clearly that the numerical solution blows up very quickly, either in almost sure sense or in mean-square sense.
			
			\begin{figure}[htbp]
				\centering
				\begin{minipage}{0.49\linewidth}
					\centering
					\includegraphics[width=0.9\linewidth]{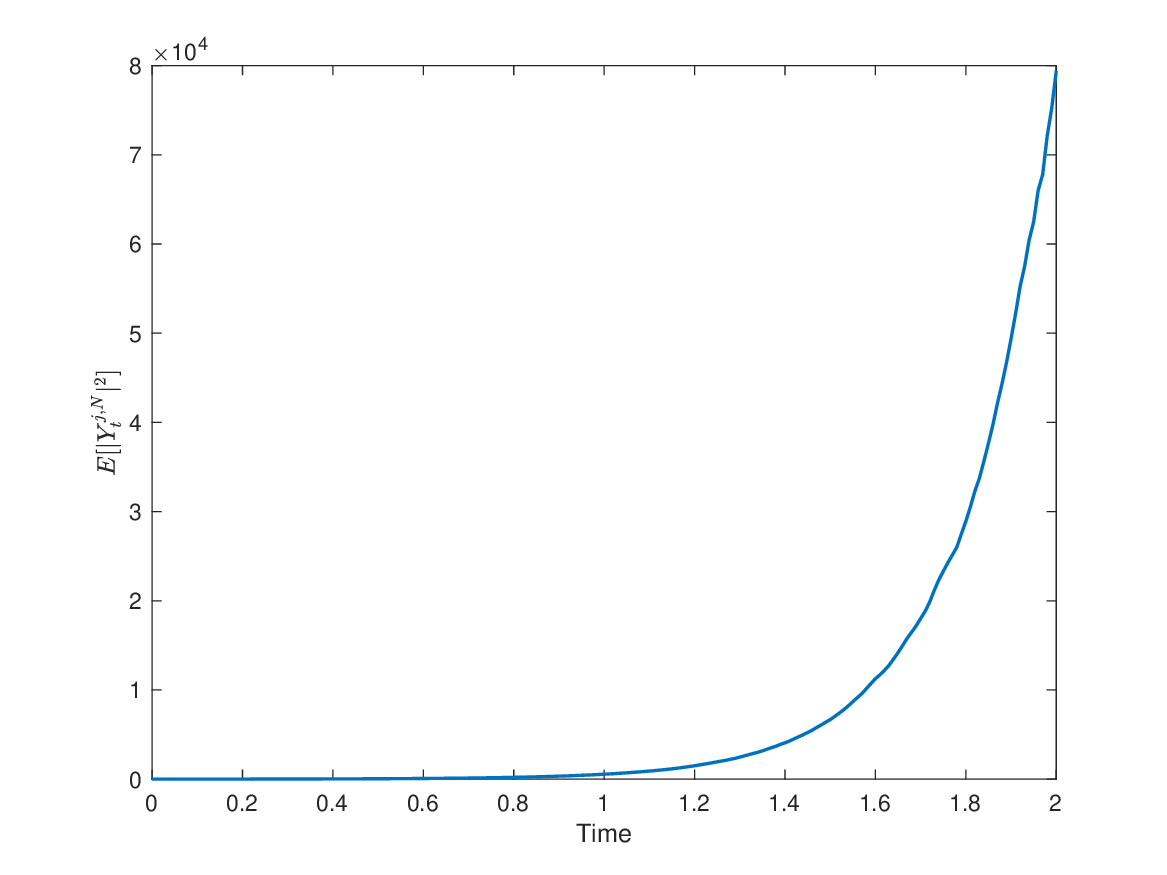}
					\caption{\label{tu5}  mean-square unstability of (\ref{exm2})}
				\end{minipage}
				\begin{minipage}{0.49\linewidth}
					\centering
					\includegraphics[width=0.9\linewidth]{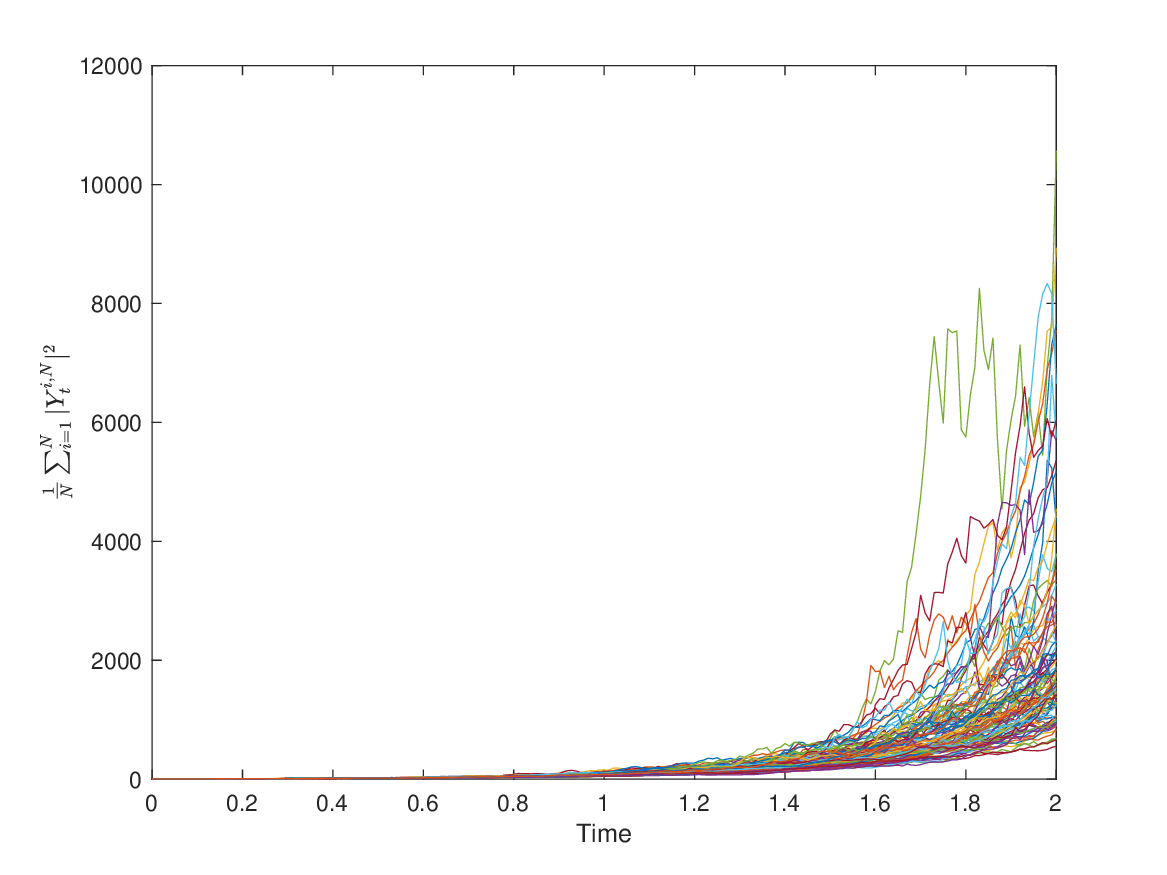}
					\caption{\label{tu6} the unstability with 100 trajectories}
				\end{minipage}
				%\qquad
				
			\end{figure}

			In order to control the system and make it stable, we add feedback control terms to  (\ref{exm2}). Then, a new system with discrete time feedback control is introduced by:
			
			\begin{equation}\label{exm4}
				dX_t=\big(2X_t+\mathbb{E}X_t-k_1X_{t_{\delta}}-k_2\mathbb{E}X_{t_{\delta}}\big) dt+X_tdW_t,
			\end{equation}
			with random initial data $X_0\sim\mathcal{N}(2,1)$. And $t_{\delta}$ is the symbol of discrete times such as $\delta,2\delta,\cdots$, where $\delta$ is the discrete-time observation gap. Here, $k_1,k_2$ are constants.

			We impose discrete-time control on the unstable system only at $t=0.05,0.1,0.15,\cdots$, namely $\delta=0.05$. To display that different values of $k_1,k_2$ will provide different controlled results, we perform numerical simulations with stepsize $\Delta=10^{-2}$ for $k_1=7,k_2=8$ and $k_1=12,k_2=10$, respectively. In the case $k_1=7,k_2=8$, Figure \ref{tu7} and Figure \ref{tu8} show that the unstable system can be controlled within a bounded range after exerting control, both in almost sure sense and mean square sense.
			In the case $k_1=12,k_2=10$, it can be observed from Figure \ref{tu9} and Figure \ref{tu10} that the unstable system can be controlled and tends to 0 as time increasing in two senses. 
			
			\begin{figure}[htbp]
				\centering
				\begin{minipage}{0.49\linewidth}
					\centering
					\includegraphics[width=0.9\linewidth]{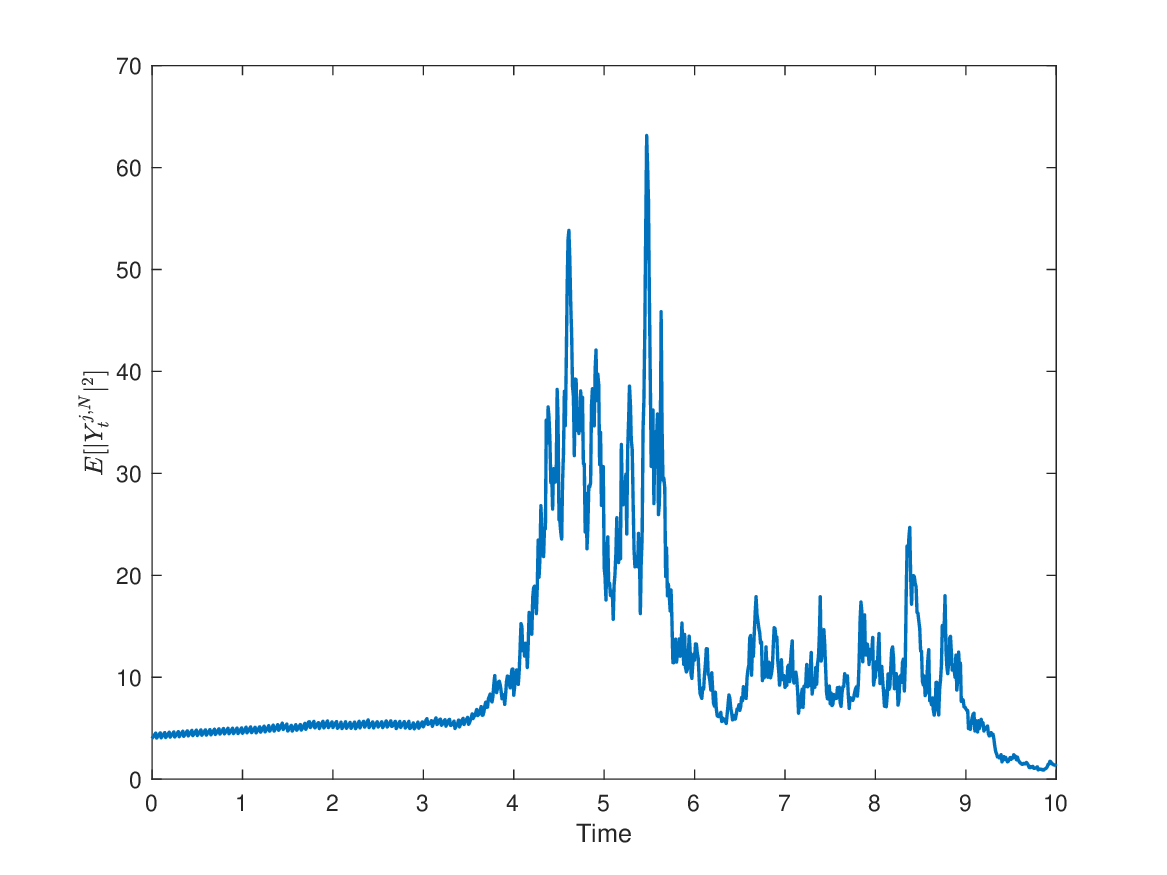}
					\caption{\label{tu7} 	mean-square stability with\\$k_1=7,k_2=8$}
				\end{minipage}
				\begin{minipage}{0.49\linewidth}
					\centering
					\includegraphics[width=0.9\linewidth]{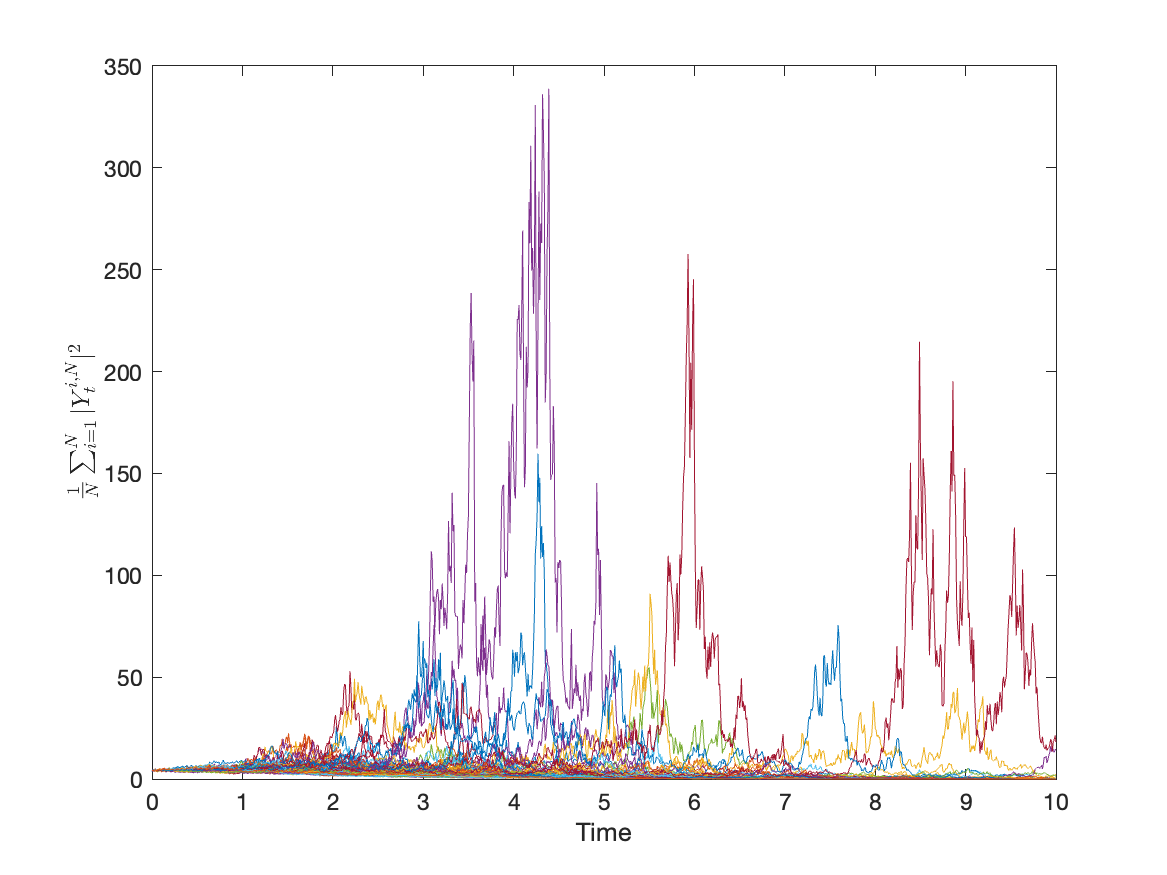}
					\caption{\label{tu8} 100 trajectories with $k_1=7,k_2=8$}
				\end{minipage}
				%\qquad

			\end{figure}
			
			\begin{figure}[htbp]
				\centering
				
				%\qquad
				\begin{minipage}{0.49\linewidth}
					\centering
					\includegraphics[width=0.9\linewidth]{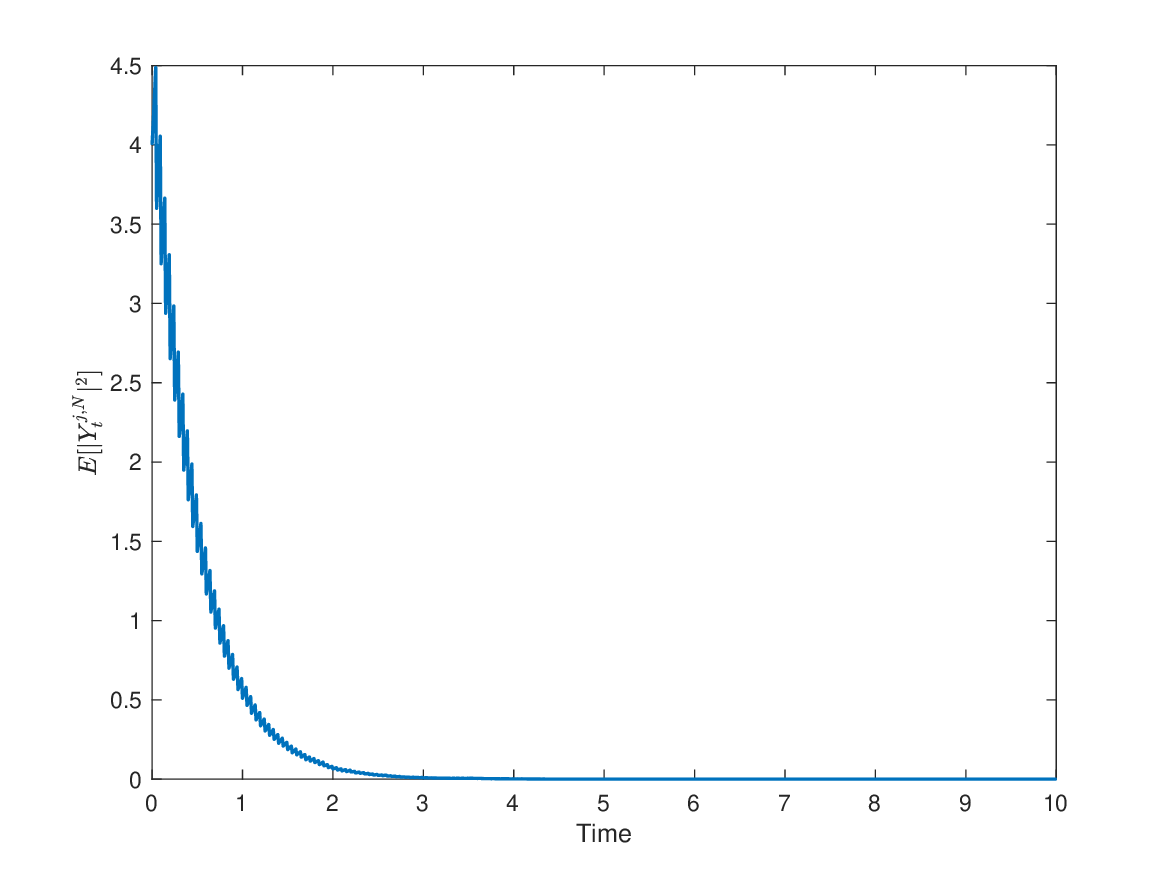}
					\caption{\label{tu9} 	mean-square stability with\\ $k_1=12,k_2=10$ }
				\end{minipage}
				\begin{minipage}{0.49\linewidth}
					\centering
					\includegraphics[width=0.9\linewidth]{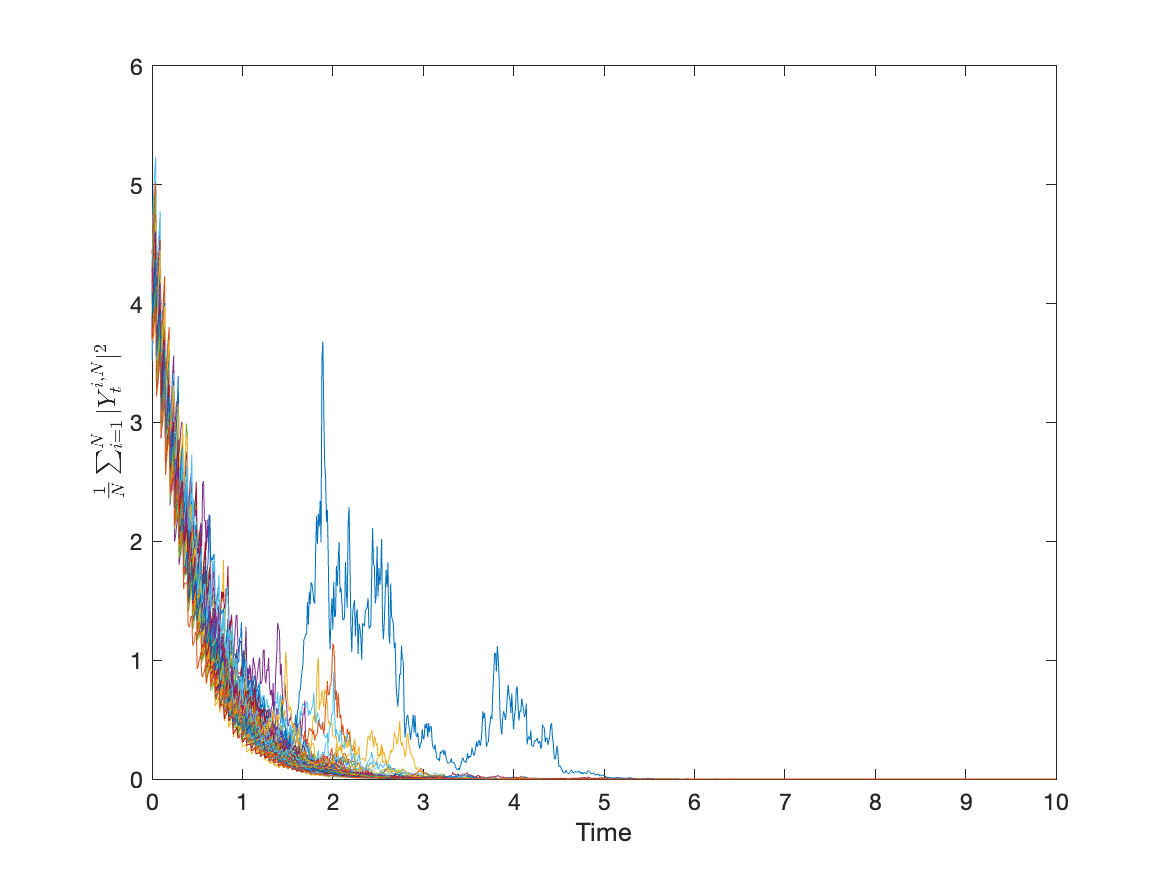}
					\caption{\label{tu10} 100 trajectories with $k_1=12,k_2=10$}
				\end{minipage}
			\end{figure}
		}
	\end{exa}
	
	The following example considers the case when the state variables in coefficients are superlinear.
	\begin{exa}
		{\rm
			Consider the following scalar SMVE
			\begin{equation}\label{example1}
				dX_t=\left(-2X_t^3-4X_t+\sin(\mathbb{E}X_t)\right)dt+\left(\rho_1X_t+\rho_2X_t^2+\sin(\mathbb{E}X_t)\right)dW_t,
			\end{equation}
			with random initial value $X_0$ which follows a normal distribution $\mathcal{N}$(0,4).
			When $\rho_1=0$, $\rho_2=1$, the coefficients of (\ref{example1}) satisfy Assumptions \ref{ass2}, \ref{ass3}, \ref{asas0}. And when $\rho_1=1$, $\rho_2=0$, the coefficients of (\ref{example1}) satisfy Assumptions \ref{ass2}, \ref{ass3}, \ref{aqsqs}.
			Based on the theories in $Section~5$, we choose $\Delta=0.004$ to simulate the corresponding interacting particle system for (\ref{example1}) with $N=300$. The mean-square stability is illustrated in $Figure$ \ref{tu03} with $\rho_1=0$, $\rho_2=1$, and the number of paths $M=500$. To show the almost sure stability, we compute and plot the average of 300 particles’ numerical solutions driven by 30 paths, with $\rho_1=1$, $\rho_2=0$ in $Figure$ \ref{tu04}.
			\begin{figure}[htbp]
				\centering	
				\begin{minipage}{0.49\linewidth}
					\centering
					\includegraphics[width=0.8\linewidth]{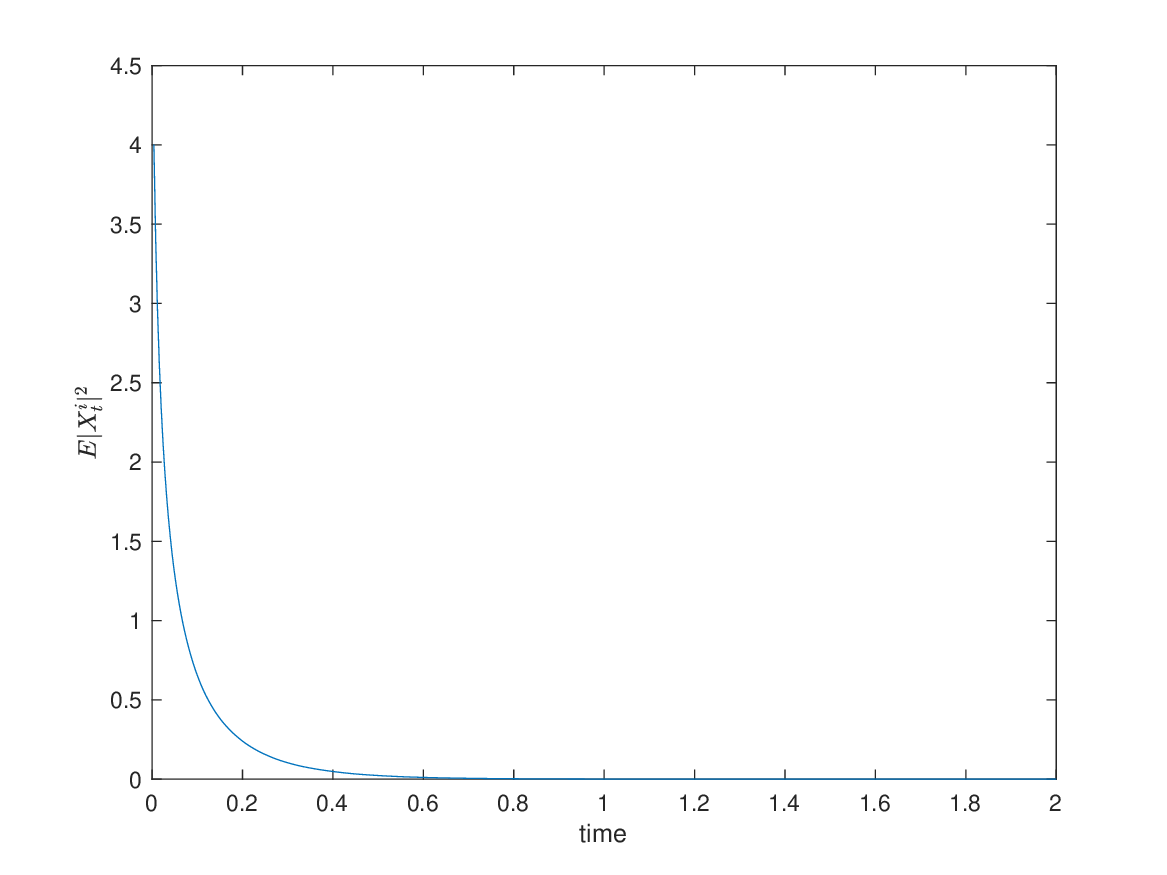}
					\caption{\label{tu03}  mean-square stability of (\ref{example1}) with $\rho_1=0$, $\rho_2=1$}
				\end{minipage}
				\begin{minipage}{0.49\linewidth}
					\centering
					\includegraphics[width=0.8\linewidth]{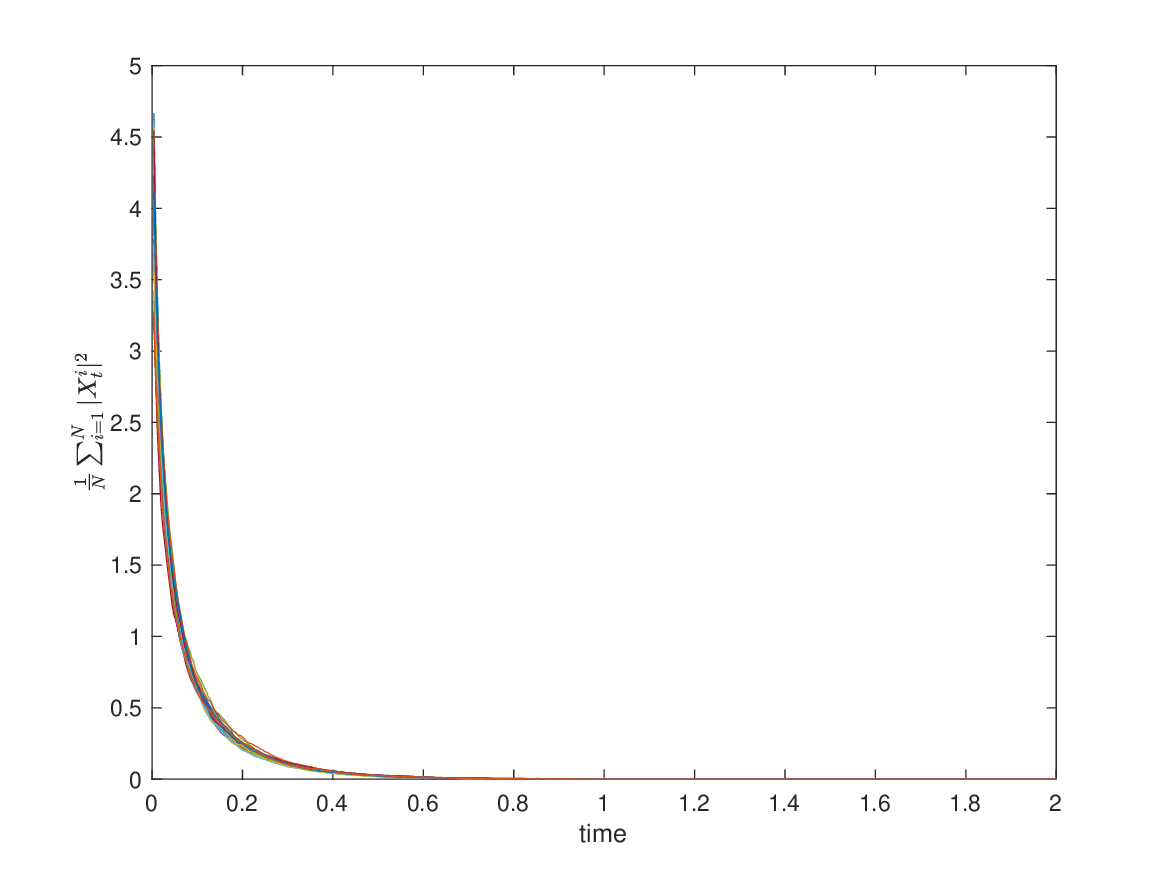}
					\caption{\label{tu04} the stability of 30 trajectories with $\rho_1=1$, $\rho_2=0$}
				\end{minipage}
				%(\ref{exm1})
			\end{figure}
			
		}
	\end{exa}
	\section*{Declaration of competing interest}
	The authors declare that they have no known competing financial interests or personal relationships that could have appeared to influence the work reported in this paper.

	\section*{Data availability}
	No data was used for the research described in the paper.

	\section*{Acknowledgements}
	
	This work is supported by the National Natural Science Foundation of China (12271368, 62373383, 62076106), the Fundamental Research Funds for the Central Universities, South-Central MinZu University (CZQ25020), and China Scholarship Council (202308310272).

\end{document}